\documentclass[11pt]{amsart}

\author{Aitor Azemar}
\address{School of Mathematics and Statistics, University of Glasgow, University Place, Glasgow G12 8QQ}
\email{Aitor.Azemar@glasgow.ac.uk}

\usepackage[top=25mm, bottom=25mm, left=30mm, right=30mm]{geometry} 
\usepackage{amssymb}

\newcommand{\E}{\mathbb{E}}
\renewcommand{\P}{\mathbb{P}}
\newcommand{\R}{\mathbb{R}}
\newcommand{\Z}{\mathbb{Z}}
\newcommand{\N}{\mathbb{N}}
\renewcommand{\S}{\mathbb{S}}

\newcommand{\inte}{\operatorname{int}}
\renewcommand{\d}{\operatorname{d}}
\newcommand{\cl}{\operatorname{cl}}
\newcommand{\nsp}{M_\infty}
\newcommand{\fsp}{M_F}
\newcommand{\BM}{D_M}

\newtheorem{theorem}{Theorem}[section]
\newtheorem{proposition}[theorem]{Proposition}
\newtheorem{lemma}[theorem]{Lemma}
\newtheorem{corollary}[theorem]{Corollary}

\theoremstyle{definition}
\newtheorem{definition}[theorem]{Definition}

\theoremstyle{remark}

\numberwithin{equation}{section}

\begin{document}

\title{Random walks on convergence groups}

%\dedicatory{}

\begin{abstract}
		We extend some properties of random walks on hyperbolic groups to random walks on convergence groups. In particular we prove that if a convergence group $G$ acts on a compact metrizable space $M$ with the convergence property then we can provide $G\cup M$ with a compact topology such that random walks on $G$ converge almost surely to points in $M$. Furthermore we prove that if $G$ is finitely generated and the random walk has finite entropy and finite logarithmic moment with respect to the word metric, then $M$, with the corresponding hitting measure, can be seen as a model for the Poisson boundary of $G$.
\end{abstract}

\maketitle

\section{Introduction}

Consider a countable group $G$ equipped with a probability measure $\mu$. We can define a random walk on $G$ by fixing a starting point and successively multiplying it by independent elements of $G$ according to the probability $\mu$. That is, by fixing a starting point $g_0$ and successively multiplying it with elements $g_i$ chosen independently according to the distribution $\mu$, thereby arriving at some element 
$$w_n:=g_0g_1\ldots g_n$$
after $n$ steps. In this paper we will study the asymptotic behaviour of such processes for a class of hyperbolic-like groups.

In the case where $G$ is a $\delta$-hyperbolic group, we can embed of $G$ into the compact space $G\cup \partial G$, where $\partial G$ is the Gromov boundary of $G$. Kaimonovich showed in \cite{Kaimonovich2} that, under mild assumptions on the measure $\mu$, the sample paths $(w_n)$ converge almost surely to points in the Gromov boundary. Furthermore, he showed that $\partial G$, together with the corresponding hitting measure $\nu$, form a model for the \emph{Poisson boundary} of $(G,\mu)$ (See \cite{Kaimonovich3} for a formal definition of this concept). That is, ($\partial G$, $\nu$) seen as a measure space encodes all the asymptotically relevant information regarding the sample paths. Similar results have been proved for many hyperbolic-like groups (see, for example, \cite{Tiozzo1} and \cite{Kaimonovich}).

We extend these results to convergence groups. Roughly speaking, these are groups that act on a space in the same way that hyperbolic groups act on their Gromov boundary. More formally, a \emph{convergence group} is a countable group $G$ acting on a compact metrizable space $M$ in such a way that for every infinite sequence $(g_n)\subset G$ of distinct elements there exists a subsequence $(g_{n_k})$ and points $a,b\in M$ such that $g_{n_k}|_{M\setminus a}$ converges to $b$ locally uniformly. That is, for every compact set $K\subset M\setminus a$ and every open neighbourhood $U$ of $b$ there is and integer $N$ such that $g_{n_k}(K)\subset U$ whenever $n_k>N$. We write that $G$ is a convergence group \emph{on} $M$ if $G$ acts on $M$ as a convergence group. Random walks on convergence groups have been previously studied by Gekhtman, Gerasimov, Potyagailo and Yang in \cite[Section 9]{ilya}.

It is fairly easy to see that hyperbolic groups act as convergence groups on their Gromov boundaries (see \cite{Bowditch} for example) so in this definition $M$ plays the role of the Gromov boundary of the group. Hence, it is reasonable to hope that the aforementioned results about random walks on Gromov hyperbolic groups extend to convergence groups if we replace $\partial G$ by $M$. Indeed, we prove the following.
\begin{theorem}\label{theo:compactification}
	Let $G$ be a discrete convergence group acting on a compact metrizable space $M$. If the action of $G$ is non-elementary and minimal, then there exists a compact topology on $G\cup M$ such that the inclusions $G\hookrightarrow G\cup M$, $M\hookrightarrow G\cup M$ are topological embeddings. For any generating measure $\mu$ on $G$, almost every sample path of the associated random walk converges to a point in $M$.
\end{theorem}
By \emph{non-elementary} action we mean that there is no invariant subset consisting of 1 or 2 points.
To prove the above result, we use a construction of Tukia from \cite{Tukia} which consists in observing that, just as $M$ is the analogue of the Gromov boundary, the space of distinct triples 
$$T=\{(a,b,c)\in M^3\mid a\neq b\neq c\neq a\}$$
is the analogue of the hyperbolic space upon which $G$ acts.
Generalizing from the case of Kleinian groups, Tukia gives a compact topology on $T\cup M$ and from here we get, in section \ref{sect:sectcompact}, a compact topology on $G\cup M$. 
To see that the random walk converges to the boundary we use similar methods as in \cite{Kaimonovich2}. 

Our next main result is seeing that $(M,\nu)$ works as a model for the Poisson boundary of $(G,\mu)$ in some cases. Specifically, we prove the following.

\begin{theorem}\label{theo:poissboundary}
	Let $G$ be a non-elementary, minimal, finitely generated convergence group on a compact, metrizable space $M$, and $\mu$ a probability measure generating $G$ with finite entropy and finite logarithmic moment with respect to the word metric. Then $(M,\nu)$ is the Poisson boundary of $(G,\mu)$, where $\nu$ is the $\mu$-stationary Borel probability measure on $M$.
\end{theorem}

One of the key ingredients that we use to prove the previous result is a theorem proven by Maher and Tiozzo in \cite{Tiozzo3}. This states that given a hyperbolic space $S$, an action of $G$ on $S$ satisfying certain properties and a measure $\mu$ on $G$ satisfying some conditions, then the Gromov boundary of $S$ together with its hitting measure forms a model for the Poisson boundary of $(G,\mu)$. To apply this result we use a quasimetric $\rho$ on $T$ which makes $(T,\rho)$ quasi-isometric to a hyperbolic space $(S,d)$ upon which $G$ acts satisfying the properties of Maher and Tiozzo's theorem. This quasimetric was introduced by Sun in \cite{binsun}. The other thing we need to apply Maher and Tiozzo's theorem is the restrictions on $\mu$, however, these are automatically satisfied whenever $G$ is finitely generated, and $\mu$ has finite entropy and finite logarithmic moment with respect to the word metric. As $(S,d)$ and $(T,\rho)$ are quasi-isometric we can identify the Gromov boundary of $(T,\rho)$ as the Poisson boundary. However, the Gromov boundary $\partial T$ of $T$ may be a complicated object, and a priori we have no direct way to relate it with $M$. Our final step is to see that $\partial T$ and $M$ together with their corresponding hitting measures are equivalent as $G$-measure spaces. We do this by building a $G$-equivariant homeomorphism between subsets with full measure.

As a corollary of the proof (in particular, of Proposition \ref{stickyzero}) we get an alternative proof of the fact that the set of conical limit points of $M$ introduced by Tukia in \cite{Tukia} has full measure under the stationary measure. A more quantitative statement of this fact has been proven by Gekhtman, Gerasimov, Potyagailo and Yang in \cite[Theorems 9.14 and 9.15]{ilya}.

Many results about a group apparently not related to random walks can be obtained by studying their asymptotic behaviour. 
For example, we say that a function $f:G\to \R$ is $\mu$\emph{-harmonic} if $f(g)=\sum_{h\in G}\mu(h)f(gh)$, that is, if the value at each point is the average (with respect to $\mu$) of the values at neighbouring points.
If $(M,\nu)$ is the Poisson boundary of $(G,\mu)$, then there exists an explicit isomorphism from $L^{\infty}(M,\nu)$ to the space of bounded $\mu$-harmonic functions on $G$. Also, using the convergence of the random walks to $M$, one can show that the action of $G$ on $(M,\nu)$ is \emph{strongly almost transitive}, that is, given any $\epsilon>0$ and $A\subset M$ with $\nu(A)>0$, there exists some $g\in G$ such that $\mu(gA)>1-\epsilon$. Having a non trivial strongly almost transitive action has interesting implications, and we refer to \cite{glasner} for a compilation of some.

\vspace{0.3cm}

I would like to thank Ursula Hamenstädt for introducing me to this topic and helping me thorough the realization of this work, Ilya Gekhtman for pointing out how to generalize the main theorem beyond finitely bounded measures, and Maxime Fortier Bourque for proofreading this paper.

\section{Preliminaries}
\subsection{Hyperbolicity and quasi-metric spaces}

Let $(X,\d)$ be a geodesic metric space, i.e., a metric space such that for any two points $a,b\in X$ there exists a path $[a,b]$ joining them, with length equal to the distance between $a$ and $b$. 
That path may not be unique, and by $[a,b]$ we mean any of them. 
Given a set $A\subset X$ and $r>0$, we will denote by $N(A,r)$ the closed $r$-neighbourhood of $A$, that is, $N(A,r)=\{x\in X\mid \d(x,A)\le r\}$. 
Given $\delta>0$, we say that $X$ is $\delta$-\emph{hyperbolic} if its triangles are $\delta$-slim, meaning that for any three points $a,b,c\in X$ and any three geodesics $[a,b]$, $[b,c]$ and $[c,a]$ we have $[a,b]\subset N([b,c]\cup [c,a], \delta)$. We say that $X$ is \emph{Gromov hyperbolic} if there exists a $\delta \ge 0$ such that it is $\delta$-hyperbolic.

In this paper we will deal with a relaxation of the notion of metric, where we soften the triangle inequality by an additive constant and allow pairs of distinct points to be at distance $0$. At large scales, this notion is indistinguishable from a metric, so many results about hyperbolicity go through. Here is a precise definition:

\begin{definition}
	Given $r\ge 0$, an \emph{$r$-quasimetric} $\rho$ on a set $Q$ is a function $\rho:Q^2\to [0,+\infty)$, satisfying $\rho(x,x)=0$, $\rho(x,y)=\rho(y,x)$ and $\rho(x,y)\le \rho(x,z)+\rho(z,y)+r$ for all $x,y,z\in Q$.
\end{definition}

A \emph{quasimetric} is an $r$-quasimetric for some $r\ge 0$. Given $s\ge 0$ and a quasimetric space $(Q,\rho)$, an \emph{$s$-geodesic segment} is a finite sequence of points $x_0,x_1,\ldots x_n$, such that
\[
|i-j|-s\le \rho(x_i,x_j)\le |i-j|+s
\]
whenever $0\le i,j\le n$.
We will also denote by $[a,b]$ any $s$-geodesic segment between $a$ and $b$, that is, such that $a=x_0$ and $b=x_n$. 
A quasimetric is a \emph{path quasimetric} if there exists $s\ge 0$ such that every pair of points can be connected by an $s$-geodesic segment. 
A path quasimetric is called \emph{hyperbolic} if, taking $s$-geodesic segments instead of geodesics, triangles are $\delta$-slim for some $\delta\ge 0$.
To ease the notation, we will always assume that $r=s=\delta$.

For an introduction on quasimetric space we refer to \cite{Bowditch2}. In there, Bowditch shows that every hyperbolic path quasimetric space is quasi-isometric to a Gromov hyperbolic space. Hence, most results about Gromov hyperbolic spaces extend to hyperbolic path quasimetric spaces. We now detail the ones we are going to use.

Given an hyperbolic path quasimetric space $(Q,\rho)$, and a point $p\in Q$, the \emph{Gromov product} on $X$ is defined by 
$$(x\cdot y)_{p}=\frac{1}{2}(\rho(p,x)+\rho(p,y)-\rho(x,y)).$$
A useful feature of this product is that $(x\cdot y)_{p}$ is equal to the distance between $p$ and any geodesic between $x$ and $y$, up to additive error. That is,
\begin{equation*}
(x\cdot y)_{p}=\rho(p,[x,y])+C(r),\label{eq:distancetogeodesic}
\end{equation*}
where we write $A=B+C(r)$ to mean that the difference between $A$ and $B$ is bounded by a constant which depends only on $r$.

Another important property we will use about the Gromov product is the \emph{reverse triangle inequality}
\[
(x\cdot y)_{p}\ge \min \{(x\cdot z)_{p},(y\cdot z)_{p}\}+C(r).
\]

An $(L,C)$\emph{-quasigeodesic} $\gamma$ is an $(L,C)$-quasi-isometric embedding of an interval $I\subset \R$ into $Q$, that is, such that for all $s$ and $t$ in $I$,
\[
\frac{1}{L}|t-s|-C\le \rho(\gamma(s),\gamma(t))\le L|t-s|+C.
\]

The following important stability result about quasigeodesics is known as Morse Lemma.
\begin{lemma}[Morse Lemma]\label{lemma:morselemma}
	Let $(Q,\rho)$ be a hyperbolic path quasimetric space and $L,C > 0$. There is $D>0$ such that for any two points $x,y\in Q$, any two $(L,C)$-quasigeodesics connecting $x$ and $y$ are contained in $D$-neighbourhoods of each other.
\end{lemma}
A proof for Gromov hyperbolic metric spaces can be found in \cite[Theorem III.1.7]{spacesnonpositivecurvature}.

The \emph{Gromov boundary} of a hyperbolic path quasimetric space $Q$, which we will denote $\partial Q$, can be defined in the same way as it is done for Gromov hyperbolic spaces. That is, given two sequences $(x_n),(y_n)\subset Q$ we say that they are equivalent, and write $(x_n)\sim (y_n)$, if for some (and hence any) $p\in Q$ we have $(x_n,y_n)_p\to\infty$. The Gromov boundary is defined as the equivalence set of sequences $(x_n)\subset Q$ such that $(x_n,x_m)_p\to \infty$ for any $p\in Q$, with the defined relation. 
The Gromov product between two elements of the boundary can be defined by
$$(x\cdot y)_{p}:=\sup \liminf_{m,n\to \infty}(x_m\cdot y_n)_{p},$$
where the supremum is taken over all sequences $(x_m)$, $(y_n)$ related to $x, y$. Furthermore, if $x\in Q$ and $y\in \partial Q$, we can use the same definition replacing the sequence $(x_m)$ by $x$.
A sequence $(x_n)\subset Q \cup \partial Q$ converges to some $y\in \partial Q$ if $(x_n \cdot y)_p$ goes to infinity for some (and hence any) $p$.
As one can see in \cite[Chapter III.3]{spacesnonpositivecurvature}, a quasi-isometry between two hyperbolic spaces induces a homeomorphism between the corresponding Gromov boundaries. The cited proof can be easily extended to quasi-isometries between path quasimetric spaces, geting the analogous result.

For any point $x\in (Q,\rho)$ and $r$-geodesic $[a,b]$, the nearest point projection $p_x$ is well defined up to a constant $K(r)$. An important property of this point is that, for any $y\in [a,b]$, 
\begin{equation}\label{eq:invtriangle}
	\rho(x,y)=\rho (x,p_x)+\rho(p_x,y) +C(r).
\end{equation}
A proof of these facts for $\delta$-hyperbolic spaces, as well as the following proposition, can be found in \cite[Section 3]{Maher}. These proofs can be easily extended to hyperbolic path quasimetric spaces with respect to the $r$-geodesics.

\begin{proposition}\label{prop:doubletriangle}
	Let $r>0$. There exists $M$, depending solely on $r$, such that for any $r-$hyperbolic path quasimetric $(Q,\rho)$, any $r-$geodesic in such space, and any $x,y\in Q$ with nearest points $p_x$ and $p_y$ respectively on $\gamma$ satisfying  $\rho(p_x,p_y)\ge M$,
	$$\rho(x,y)=\rho(x,p_x)+\rho(p_x,p_y)+\rho(p_y,y)+C(r).$$
\end{proposition}

If $G$ is a group acting by isometries on $Q$, we say $g\in G$ is a \emph{loxodromic} element if the map $\Z\to Q$, $n\to g^{n}x$ is an $(L(x),C(x))$-quasi-isometric embedding for some (equivalently, any) $x\in Q$, that is, $t\to g^{\left \lfloor{t}\right \rfloor}x$ is a quasi-geodesic. Of interest to us will be the following property of these elements, well known when $Q$ is a proper hyperbolic metric space. %\colorbox{red}{PUT REFERENCE}
\begin{proposition}
	\label{prop:loxodromic}
	Let $G$ be a group acting by isometries on a hyperbolic path quasimetric space $(Q,\rho)$, and let $g$ be a loxodromic element. Then there exists $N\in \N$ and $M>0$ such that, for any $n\in \N$, $\inf_{x\in Q}\rho(x,g^{nN}x)\ge nM$.
\end{proposition}

\begin{proof}
	Fix $s\in Q$. By definition of loxodromic element, the set $\{g^ns \mid n\in \Z\}$ is an $(L,C)$-quasigeodesic for some $L,C$. Given $x\in Q$, consider $k\in \Z$ such that one of the closest point projections of $x$ to $\{g^ns \mid n\in \Z\}$ is $g^k s$. 
	By definition of nearest point projection we have that $\rho(x,g^ks)\le \rho(x,g^n s)$ for all $n\in \Z$. The group $G$ acts by isometries and the set $\{g^ns\}$ is $g$-invariant, so for any $m\in\Z$ the nearest point projection of $g^mx$ can be chosen to be $g^{m+k}s$. Consider now the $r$-geodesic $\gamma$ between $g^{k}s$ and $g^{m+k}s$, and the projections $p_x$ and $p_{g^mx}$ of $x$ and $g^mx$ to $\gamma$. By Morse Lemma, there is a constant $D$ such that the geodesic $\gamma$ is at $D$ distance from the points $\{g^ns,k\le n \le k+m\}$, so by the triangle inequality,
	\begin{align*}
	\rho(x,p_x)&\ge \rho(x,\{g^ns\mid k\le n \le m+k\})-\rho(\{g^ns\mid k\le n \le m+k\},p_x)-r\\
	&\ge\rho(x,g^{k}s)-r-D.
	\end{align*}
	Adding equation \eqref{eq:invtriangle}, we obtain
	$$\rho(p_x,g^{k}s)=\rho(x,g^{k}s)-\rho(x,p_x)+C(r)\le D+r+C(r).$$
	The same result can be obtained in the same way for the distance between $g^{m+k}s$ and $p_{g^mx}$. Hence, we get
	$$\rho(p_x,p_{g^mx})\ge \rho(g^{k}s,g^{m+k}s)-2K-2r+C(r)\ge \frac{k}{L}-C-2D-2r+C(r).$$
	Since $L$, $C$ and $K$ depend only on $s$, we can take $m$ such that $\frac{m}{L}-C-2D-2r+C(r)$ is big enough so that we can apply Proposition \ref{prop:doubletriangle}. Hence we get, for any $x\in Q$,
	$$\d(x,g^mx)=\d(x,p_x)+\d(p_x,p_{g^mx})+\d(p_{g^mx},g^mx)+C(r)\ge \frac{m}{L}-C-2D-2r+C(r),$$
	so the proposition is satisfied for $0<M<\frac{1}{L}$ and $N>L(M+C+2D+2r)$ (and big enough such that Proposition \ref{prop:doubletriangle} applies).
\end{proof}

We say that an element $h\in G$ is
\emph{weakly properly discontinuous} (WPD) if for every $s\in S$ and $\epsilon >0$ there exists $K\in \N$ such that 
$$|\{f\in G\mid \d(s,fs)<\epsilon \text{ and } \d(h^{K}s,fh^{K}s)<\epsilon\}|<\infty.$$

Finally, we recall that a group is called \emph{hyperbolic} if it is finitely generated, and any Cayley graph obtained from a finite set of generators it hyperbolic in the word metric. Since two Cayley graphs generated by different finite sets of generators are quasi-isometric, and the hyperbolicity property and Gromov boundary are invariant by quasi-isometries, the notion of hyperbolic group is well defined and one can talk about the Gromov boundary of the group. By the Švarc–Milnor lemma, any group acting by isometries, properly discontinuously and cocompactly on a proper hyperbolic space is hyperbolic.

\subsection{Random Walks}\label{sect:randomwalks}
Let $G$ be a discrete group and $\mu$ a probability measure on $G$. The \emph{step space} $\Omega:=G^\N$ is the space of infinite sequences of group elements, which we consider as a probability space with the product measure $\P:=\mu^\N$. We will denote \emph{random walk} on $G$ starting at $g_0$ the stochastic process (indexed by $\N\cup 0$) obtained by associating to each $n$, the $G$-valued random variable $w_n:\Omega\to G$ defined by
$$(g_1,g_2,\ldots)\mapsto w_n:=g_0g_1\ldots g_n.$$
In other words, a random walk on $G$ is a time homogeneous Markov chain with transition probabilities given by $p(g,h)=\mu(g^{-1}h)$. Our random walks will always start at the neutral element, that is, $g_0=e$.

In this paper, the group $G$ will act by isometries on some metric space $(X,\d)$, and we will be interested in the process we get by applying the random walk to some starting point $x\in X$, i.e., in the process $(w_n x)_{n\in \N}$. 
We will refer to this new process as random walk on $X$ (generated by $(G,\mu)$). This can also be seen as the projection of the random walk on $G$ to $X$.

We will be interested in the asymptotic behavior of the random walks, 
in particular, whether they converge to some boundary at infinity, and in which way they converge. 
Assume $G$ can be embedded into a $G$-space of the form $G\cup B$ (that is, a topological space upon which $G$ acts by homeomorphisms), and that for almost every $\omega \in \Omega$, the sample path $(w_n(\omega))$ converges to some point $w_\infty(\omega)\in B$. 
We say that a measure $\nu$ on $B$ is $\mu$-stationary if for any measurable $A\subset B$ we have $\nu(A)=\sum_{g\in G}\mu(g)g\nu(A)$, where $g\nu(A):=\nu(g^{-1}A)$.
Furstenberg shows in \cite{Furstenberg} that the resulting hitting measure $\nu$ in $B$ is $\mu$-stationary and the measure $w_n\nu$ converges in the weak-* topology to a point measure. 
With this in mind, Furstenberg defines the following.
\begin{definition}
	Let $G$ be a group acting on a measurable $G$-space $(B,\nu)$ and $\mu$ a measure on $G$. Then $(B,\nu)$ is a \emph{$\mu$-boundary} (or \emph{Furstenberg boundary}) of $(G,\mu)$ if
	\begin{enumerate}
		\item $\nu$ is a $\mu$ stationary probability measure;
		\item for almost every sample path $(w_n)$, the sequence of measures $(w_n \nu)$ converges weakly to a $\delta$-measure. 
	\end{enumerate}
\end{definition}
Furstenberg also shows that whenever $(B,\nu)$ is a $\mu$-boundary we can endow $G\cup B$ with a topology such that the sample paths of the random walks converge almost surely to points in the boundary.
However, the inclusion $G\hookrightarrow G\cup B$ might not be an embedding.

If we consider the $\mu$ boundaries of a group as measure spaces, we can establish a partial order between them by considering $(B_1,\mu_1)\ge (B_2,\mu_2)$ if there exists a $G$-equivariant map $f:B_1\to B_2$ such that $(B_2,\mu_2)=(f(B_1),f_*\mu_1)$. 
Furstenberg also shows that, up to the equivalence given by the order relation, there exists a unique maximal $\mu$-boundary, called the \emph{Poisson boundary}.

Using the strip criterion developed by Kaimonovich in \cite[Theorem 6.4]{Kaimonovich}, and assuming that the measure $\mu$ has finite logarithmic moment (that is, $\sum_{g\in G}\mu(g)|\log(\d(x,gx))|<\infty$), Maher and Tiozzo prove in \cite{Tiozzo3} the following theorem, where they determine the Poisson boundary for wide a variety of groups.

\begin{theorem}[Maher and Tiozzo]\label{theo:poissonacylindrical}
	Let $G$ be a countable group which acts by isometries on a hyperbolic metric space $(X,d)$, and let $\mu$ be a non-elementary probability measure on $G$ with finite logarithmic moment and finite entropy.
	Suppose that there exists at least one WPD element $h$ in the semigroup
	generated by the support of $\mu$. Then the Gromov boundary of $X$ with the hitting
	measure is a model for the Poisson boundary of the random walk $(G,\mu)$.
\end{theorem}

Just as it happens with hyperbolic groups, the Gromov boundary, together with a stationary measure, is a model for the Poisson boundary. 
This result is an improvement of their previous result, proven in \cite{Tiozzo2}, where they required the action to be acylindrical.

\newcommand{\something}{*}
\newcommand{\outsideset}{O}
\newcommand{\insideset}{I}

\subsection{Convergence Groups}\label{sect:convergencegroups}

The notion of convergence group was originally introduced by Gehring and Martin in \cite{Kleinian}, where they axiomatize the dynamical properties of Kleinian groups acting on the Gromov boundary of $\mathbb{H}^n$. In particular, they give the following definition.
\begin{definition}
	Let $G$ be a discrete countable group acting on a compact metrizable space $M$. $G$ is called a \emph{convergence} group if for every infinite sequence $(g_n)\subset G$ of distinct elements, there exists a subsequence $(g_{n_k})$ and points $a,b\in M$ such that $g_{n_k}|_{M\setminus a}$ converges to $b$ locally uniformly, that is, for every compact set $K\subset M\setminus a$, and every neighbourhood $U$ of $b$, there is $N$ such that $g_{n_k}(K)\subset U$ whenever $n_k>N$.
\end{definition}

The points $a$ and $b$ are respectively called the \emph{repelling} and \emph{attracting} points of the subsequence $(g_{n_k})$. 
We say that $G$ is a convergence group on $M$ if it acts on $M$ as a convergence group.

Convergence groups appear naturally when dealing with groups acting on hyperbolic spaces. 
Indeed, Bowditch proves the following result in \cite{Bowditch}.
\begin{proposition}[Bowditch]\label{hypgroupsareconve}
	Let $G$ be a group acting by isometries and properly discontinuously on a proper hyperbolic space $X$. Then $G$ is a convergence group on the Gromov boundary $\partial X$. In particular, all hyperbolic groups are convergence groups.
\end{proposition}

%Thorough this section, $G$ will represent a fixed group acting on a fixed metrizable space $M$ as a convergence group. 
Adapting the definition for hyperbolic spaces, we say $G$ is \emph{non-elementary} if there is no invariant subset of $M$ consisting of at most 2 points.
We say that the action is \emph{minimal} if $M$ has no proper closed invariant set.
We will always assume that the action of $G$ is non-elementary and minimal.
Note that while requiring the action to be non-elementary is a restriction on $G$, the minimality is not, as we can always take a subset of $M$ such that the restricted action of $G$ is minimal.

\subsubsection{Kleinian groups and the space of distinct triples}\label{seckleinian}

A \emph{Kleinian group} is a discrete group of Möbius transformations of the $n$-sphere $\S^n$.
The action can be extended to act on the $(n+1)$-ball $B^{n+1}$, and the ball can be equipped with a hyperbolic metric $d_H$ such that the extension of the Möbius transformations act by isometries. 
Hence, Kleinian groups are discrete groups acting by isometries on the hyperbolic space $(B^{n+1},d_{H})$. 
Since the extension is always properly discontinuous, by Proposition \ref{hypgroupsareconve} Kleinian groups act as convergence groups on $\partial B^{n+1}=\S^n$.

In this case where $M=\S^n$ is the Gromov boundary of some hyperbolic space, we can define a map from the space of distinct triples
$$ T:=\{(a,b,c)\in M^3\mid a\neq b\neq c\neq a\}$$
to $M$ by $p(a,b,c):=z$, where $z$ is the projection of the boundary point $c$ on the unique geodesic between $a$ and $b$. Endowing $T$ with the induced topology, we have that the diagonal action by $G$ defined by $g(a,b,c)=(ga,gb,gc)$ is continuous.
It is easy to see that $p$ commutes with $G$, that the preimage of a point under $p$ is compact, and that given two points in $B^{n+1}$, their preimages by $p$ are homeomorphic.
Therefore, in the case of Kleinian groups, $T$ can be seen as a bigger version of $B^{n+1}$. 
As Tukia points out in \cite{Tukia}, $T$ works as a rough equivalent to the hyperbolic space for convergence groups. 
For example, Bowditch shows in \cite[Lemma 1.1]{Bowditch} that the action of $G$ on $M$ is a convergence action if and only if the induced action on $T$ is properly discontinuous, bearing some similarity to Proposition \ref{hypgroupsareconve}. 
Tukkia also pastes $M$ to $T$ in an analogous way to that of the Gromov boundary. 
Before explaining how the pasting goes, it is convenient to see the following lemma.

\begin{lemma}\label{lemma:kleinanconvtobound}
	Let $(x_n)\subset B^{d+1}$ be a sequence such that $x_n\to \lambda\in \partial B^{d+1}=\S^d$. Then, given a neighborhood $U$ of $\lambda$ in $\S^n$, there exists $n_0$ such that for all $n\ge n_0$, every member of $p^{-1}(x_n)$ has at least two components inside $U$.
\end{lemma}
\begin{proof}
	Fix $x\in B^{n+1}$ and consider $R>0$ such that the neighborhood $V(\lambda,r):=\{y\in \overline{B}^{n+1}\mid(\lambda,y)_x>R\}$ of $\lambda$ in $B^{d+1}$ satisfies $V(\lambda,R)\cap \S^n\subset U$. Since $x_n\to \lambda$, given $C>0$, there exists $n_0$ such that $x_n\subset U(\lambda,R+C)$ for all $n\ge n_0$. Fix then $n\ge n_0$ and $(a,b,c)\in p^{-1}(x_n)$. Consider the Gromov products $(\lambda\cdot a)_x$, $(\lambda\cdot b)_x$ and $(\lambda\cdot c)_x$. If none of them is smaller than $R$, we are done. Assume $(\lambda\cdot c)_x$ to be the smallest, and that it is smaller than $R$. By the reverse triangle inequality, we have $(\lambda\cdot c)_x\ge \min((c\cdot x_n)_x,(x_n\cdot \lambda)_x)$, so 
	$$\d(x,[c,x_n])=(c\cdot x_n)_x+C(\delta)\le (\lambda\cdot c)_x+C(\delta)\le R+C(\delta),$$
	where $\delta$ is the hyperbolicity constant associated to the hyperbolic space.
	The geodesics $[a,b]$ and $[c,x_n]$ meet orthogonaly at $x_n$, so given any $y\in [a,b]$ the closest point projection of $y$ to the geodesic $[c,x_n]$ is $x_n$. Recall that $(x_n\cdot \lambda)_x\ge R+C$ and $\d(x,x_n)\ge R+C$, so if $q$ is the projection of $x$ to $[c,x_n]$, then $\d(q,x_n)\ge \d(x,x_n)-\d(x,q)\ge C$. Taking $C$ big enough so Proposition \ref{prop:doubletriangle} applies, we get $\d(x,y)\ge R+C+O(\delta)$. Hence, $(a\cdot x_n)_x,(x_n\cdot b)_x\ge R+C+O(\delta)$. Finally, $(a\cdot \lambda)_x\ge\min((a\cdot x_n)_x,(x_n\cdot \lambda)_x)\ge R+C+O(\delta)$, so $a\in U$, and similarly $b\in U$. If $(c\cdot \lambda)_x$ is not the smallest, we can use that $p(b,c,a)$ and $p(c,a,b)$ are at a bounded distance from $p(a,b,c)$, and take $C$ a little bigger.
\end{proof}

This lemma shows that the notion of convergence to the boundary on $\overline{B}^{n+1}$ can be translated to $T\cup M$ via the following neighbourhoods.
Given $U\subset M$ an open set, we define the associated set on $T\cup M$ by
$$\widetilde{U}=\{x\in T\mid x \text{ has at least two components in } U\}\cup U.$$
Adding to these sets the open sets of $T$ we get a basis for a unique topology on $T\cup M$.
From Lemma \ref{lemma:kleinanconvtobound}, if $x_n\to \lambda\in \S^d$ in $\overline{B}^{d+1}$ then any sequence of preimages $\tilde{x}_n\in p^{-1}(x_n)$ will also converge to the same $\lambda\in M=\S^d$ in the above topology on $T\cup M$. 
Conversely, if $(\tilde{y}_n)$ converges to $\lambda$ in $T\cup \S^n$ then $(p(\tilde{y}_n))$ will be $C(\delta)$-close to a geodesic with two endpoints that are close to $\lambda$ so it also converges to $\lambda$ in $\overline{B}^{n+1}$. 
Therefore, just as $T$ can be regarded as a rough equivalent of the hyperbolic space, $M$ can be seen as a rough equivalent of its Gromov boundary, and this way of pasting them together works as an equivalent of Gromov's topology.

\subsubsection{The metric of the space of triples}\label{sect:metricspacetriples}

Sun shows in \cite{binsun} that the analogy from the last section can be taken a step further by actually endowing $T$ with a hyperbolic path quasimetric $\rho$ in such a way that $G$ acts by isometries on $(T,\rho)$. 
The quasimetric is based on a construction done by Bowditch in \cite{Bowditch2}. To define the quasimetric we first have to introduce some concepts.

\begin{definition}
	An \emph{annulus} $A$ is an ordered pair $(A^-,A^+)$ of disjoint closed subsets of $M$ such that $M\setminus (A^-\cup A^+)\neq \emptyset$. 
	A set of annuli $\mathcal{A}$ is an \emph{annulus system}. It is \emph{symmetric} if $A\in\mathcal{A}$ implies $-A:=(A^+,A^-)\in \mathcal{A}$.
\end{definition}

For a $g\in G$, we denote $gA$ the annulus $(gA^-,gA^+)$.

For any subset $K\subset M$ we define the relations $K<A$ if $K\subset \inte A^-$ and $A<K$ if $K\subset \inte A^+$. 
If $B$ is another annulus we write $A<B$ if $\inte A^+\cup \inte B^-=M$. 
Since $B^+\subset (B^-)^c$ this implies $A^+\supset B^+$ and $A^-\subset B^-$. 

For an annulus system $\mathcal{A}$ on $M$ and $K,L\subset M$ we define $(K|L)=n\in\{0,1,\ldots,\infty\}$ where $n$ is the maximal number of annuli $A_i$ in $\mathcal{A}$ such that we can build the chain 
\[
K<A_1<A_2<\ldots <A_n<L.
\]
This gives us two sequences of inclusions, $K\subset A_1^-\subset A_2^-\ldots \subset A_n^-\subset L^c$ and $K^c\supset A_1^+\supset\ldots \supset A_n^+\supset L$. For finite sets, we drop the braces and write $(a,b|c,d)$ to mean $(\{a,b\}|\{c,d\}).$ The function $(\cdot,\cdot|\cdot,\cdot)$ from $M^4$ to $\N \cup \infty$ we just defined is called the \emph{crossratio}. With all of this we can define the function which will give us the quasimetric.

\begin{definition}
	Given an annulus system $\mathcal{A}$ on $M$, define the function $\rho:T^2\to [0,\infty]$ by
	$$\rho((x^1,x^2,x^3),(y^1,y^2,y^3)):=\max\{(x^i,x^j|y^k,y^l) : i\neq j,k\neq l\}.$$
\end{definition}

In \cite{Bowditch2} it is shown that if the annulus system is $G$-invariant, symmetric and such that $\mathcal{A}/G$ is finite then the previous function takes values in $[0,\infty)$ and is a $G$-invariant hyperbolic path quasimetric.
The geometric realization of the graph obtained by considering the points of $T$ as vertices and joining them by edges whenever their $\rho$ distances are smaller than some number $s$, defined in Sun's paper \cite{binsun}, is a hyperbolic metric space $(S,\rho')$.
The action induced by $G$ on this space is isometric and the inclusion $T\hookrightarrow  S$ is a $G$-equivariant quasi-isometry.

The remaining step is to choose a convenient annulus system. The following result, found in \cite{Tukia2}, will play an important role in that.
\begin{theorem}[Tukia]\label{theo:preloxodromic}
	If $G$ is a non-elementary convergence group on $M$, then there is an element $g\in G$ such that $g$ fixes two distinct points $a,b$ and such that $g^n|_{M\setminus a}$ converges to $b$ locally uniformly as $n\to \infty$.
\end{theorem}
Consider such an element $g$ and two closed sets $A^-$, $A^+$ such that $A^-\cap A^+=\emptyset$ and $a\in \inte A^-$, $b\in \inte A^+$. These exist, since $M$ is metrizable. Fix the annulus $A:=\{A^-,A^+\}$ and the annuli system generated by $A$
$$\mathcal{A}:=\{g(\pm A) \mid  g\in G\}.$$
Then, $\mathcal{A}$ is symmetric and $\mathcal{A}/G$ is finite. 
Sun proves in \cite{binsun} that there exists some $N\in \N$ such that $h:=g^N$ is a loxodromic and WPD element.

\section{Random walks on Convergence Groups}

Let $G$ be a non-elementary convergence group on a compact metrizable space $M$. Given a probability measure $\mu$ on $G$ we have, as shown in \cite[Theorems 9.7 and 9.8]{ilya}, that there exists a unique $\mu$-stationary measure $\nu$ on $M$ and that $(M,\nu)$ form a $\mu$-boundary of $G$. 

It will be useful for us to know the behavior of the random walk on the space $T$. 
In particular, we will require in the following result.

\begin{proposition}\label{prop:convtobound}
	Let $G$ be a non-elementary, minimal, convergence group on a compact metrizable space $M$ and let $\mu$ be a probability measure on $G$ such that its support generates $G$. 
	Then for any $x\in T$ the sample paths $(w_nx)$ of the associated random walk converge almost surely to $M$, where the topology on $T\cup M$ is the one defined after Lemma \ref{lemma:kleinanconvtobound}.
\end{proposition}
\begin{proof}
	As $(M,\nu)$ is a $\mu$-boundary, $w_n\nu$ converges to $\delta_{p(w)}$ for some $p(w)\in M$ almost surely. 
	Assume we have $w\in \Omega$ such that $w_n\nu$ converges to $\delta_p$ but $w_n x$ does not converge to $p$. 
	Then, there exists a neighborhood $U$ of $p$ in $M$ such that $w_{n_k} x\notin \widetilde{U}$ for infinitely many $n_k$. 
	We have $w_{n_k}\nu \to \delta_p$ so if $(w_{n_k})$ has finitely many elements there exists some $s$ such that $w_{n_{s}} \nu=\delta_p$.
	Then, $\nu=\delta_{w_{n_s}^{-1}p}$ which can not be as $\nu$ is non atomic, as proven in  \cite[Theorem 9.4]{ilya}. 
	Hence, $(w_{n_k})$ has infinitely many elements and we may take a convergent subsequence relabeled $(w_i)$. 
	If the attracting point $p'$ of $w_i$ is $p$ then we may take a neighborhood $V$ around the repelling point small enough so its closure contains at most one component of $x$.
	By the definition of convergence action there exists $i_0$ such that for $i>i_0$ we have $w_i(M\setminus V)\subset U$, so for any neighborhood $U$ of $p$ we have $w_ix\in \widetilde{U}$, and $w_ix \to p$.
	
	Assume then that $p'$ is not $p$.
	By definition we have $w_i\nu\to \delta_{p'}$ if for every continuous function $f$ on $M$, 
	$$\int_{x\in M} f(x)w_i\nu\to \int_{x\in M}f(x)\delta_{p'}=f(p').$$ 
	Fix such an $f$ and choose an arbitrary $\epsilon >0$.
	By continuity we can consider a neighborhood $U$ of $p'$ such that $|f(x)-f(p')|<\epsilon$ for all $x\in U$.
	Furthermore, $\nu$ is non-atomic and Borel, and $M$ is metrizable, so $\nu$ is regular and we can consider a neighborhood $V$ around the repelling point $a$ such that $\nu(V)\le \epsilon$.
	By the convergence property there exists $i_0$ such that for $i\ge i_0$ we have $w_{i} M\setminus V\subset U$.
	Hence, 
	\begin{align*}
	\left|\int_{x\in M} f(x)w_i\nu-f(p')\right|&\le\int_{x\in M} |f(x)-f(p')|w_i\nu=\int_{x\in M} |f(w_i x)-f(p')|\nu \\
	&=\int_{x\in M\setminus V} |f(w_ix)-f(p')|\nu+\int_{x\in V} |f(w_ix)-f(p')|\nu \\
	&\le \nu(M\setminus V)\sup_{x\in U}|f(x)-f(p')|+\nu(V)\sup_{x\in M}|f(x)-f(p')|\le \epsilon+\epsilon 2K,
	\end{align*}
	where $K$ is the maximum value of $f$.
	Hence, $w_{i}\nu$ converges to $\delta_{p'} \neq \delta_p$, and we have the contradiction.
\end{proof}

\subsection{The Gromov boundary of $(T,\rho)$}\label{sect:gromboundrayofT}

Applying Theorem \ref{theo:poissonacylindrical} to Sun's construction we get that the Gromov boundary $\partial T$ of $(T,\rho)$, together with the corresponding hitting measure, is a model for the Poisson boundary of our walk.
We would like to relate this with the $\mu$-boundary given by $(M,\nu)$. 
For this, we compare the two possible boundaries using the relations between the corresponding pastings to $T$, that is, the pasting on $T\cup M$ corresponding to Tukia's topology, and the one on $T\cup \partial T$ corresponding to Gromov's topology.
%We have now two possible boundaries for $T$, the one given by the Gromov boundary for the quasi-metric $\rho$, $T\cup \partial T$, and the one given by Tukia's topology, $T\cup M$. In this section we will determine what relation do these two boundaries have. 

We analyze sequences of $T$ converging to points in the boundaries. 
In particular we first see that given a sequence $(x_n)\subset T$ such that $x_n\to \lambda\in \partial T$ there exists a $p\in M$, which only depends on $\lambda$, such that $x_n\to p$. 
This gives us an application $\phi:\partial T\to M$, which actually is $G$-equivariant and continuous. 
It is not possible to repeat the process in the other direction as some sequences converging to some points in $M$ might be bounded in $(T,\rho)$ and hence not converging to any point in $\partial T$. 
However, we are able to create a continuous inverse restricted to the points of $M$ where such sequences do not exist.
Therefore, we get a homeomorphism between that subset and the corresponding subset of $\partial T$.

Before starting to build $\phi$, we prove some lemmas that we will use on many occasions.

\begin{lemma}\label{lemma:finiteamount}
	Let $G$ be a convergence group on $M$. Also let $A^-,A^+$ be two disjoint closed sets and $B_1,C_1,B_2,C_2\subset M$ be such that $B_i,C_i$ have separating neighborhoods (i.e., there exists open sets such that $B_i\subset V_i$, $C_i\subset U_i$, $\cl V_i\cap \cl U_i=\emptyset$). 
	Then
	$$|\{g\in G \mid gA^-\cap B_1\neq \emptyset,gA^-\cap  C_1\neq \emptyset,gA^+\cap B_2\neq \emptyset,gA^+\cap  C_2\neq \emptyset\}|<\infty.$$
\end{lemma}
\begin{proof}
	Assume we have infinitely many elements on that set. 
	Take a sequence $(g_n)_{n\in \N}$ (with $g_i\neq g_j$ for $i\neq j$) and a convergent subsequence relabeled as $(g_n)$. 
	Assume that the repelling point of the subsequence is not in $A^-$.
	Since $g_{n}A^-$ intersects $B_1$ and $C_1$ we can choose an open set $W$ around the repelling point, not intersecting $A^-$, and $n_0$ big enough such that $g_{n_0}(M-W)\subset O$, where $O$ is an arbitrary open set around the attracting point.
	Therefore $g_{n_0} A^-\subset O$, so taking an open set small enough such that $O\cap V_1=\emptyset$ or $O\cap U_1=\emptyset$ (which we can do, since the closures don't intersect) we get a contradiction to the definition of the set.
	Hence, the repelling point has to be in $A^-$.
	However, doing the same reasoning for $i=2$, we get that the repelling point also has to be in $A^+$, which is not possible since $A^-\cap A^+=\emptyset$.
\end{proof}

This lemma allows us to get some relations between the values of crossratios.
\begin{lemma}\label{lemma:finitedistance}
	Let $B_1,B_2,C_1$ and $C_2$ as in Lemma \ref{lemma:finiteamount}. There exists a $K=K(B_1,B_2,C_1,C_2)<\infty$ such that if $x\in B_1$, $y\in C_1$, $z\in B_2$ and $t\in C_2$, then $(x,y|z,t)\le K$.
\end{lemma}
\begin{proof}
	Assume the conditions of the lemma are satisfied. We have that
	\begin{align*}
	&|\{A_j\in \mathcal{A}\mid A_j^-\cap B_1\neq \emptyset,A_j^-\cap C_1\neq \emptyset,A_j^+\cap B_2\neq \emptyset,A_j^+\cap C_2\neq \emptyset\}|\le\\
	&|\{g\in G \mid gA^-\cap B_1\neq \emptyset,gA^-\cap C_1\neq \emptyset,gA^+\cap B_2\neq \emptyset,gA^+\cap C_2\neq \emptyset\}|+\\
	&|\{g\in G \mid gA^+\cap B_1\neq \emptyset,gA^+\cap C_1\neq \emptyset,gA^-\cap B_2\neq \emptyset,gA^-\cap C_2\neq \emptyset\}\}|= K < \infty,
	\end{align*}
	where the last inequality follows from Lemma \ref{lemma:finiteamount}. So, since our annuli system $\mathcal{A}$ consists of annuli of the form $(gA^-,gA^+)$ and $(gA^+,gA^-)$, the maximum chain between $\{x,y\}$ and $\{z,t\}$ is smaller than $K$.
\end{proof}
A simple application of this result is as follows. 
Given sequences $(x_n),(y_n),(z_n)$ and $(t_n)$  in $M$ converging to $x,y,z$ and $t$ with $x\neq y$ and $z\neq t$, there exists some $n_0$ such that, for $n\ge n_0$ we have $(x_n,y_n|z_n,t_n)\le K(x,y,z,y)<\infty$.

Another application is as follows.
\begin{lemma}\label{lemma:boundedcrossratiobychangingpoints}
	Let $I\subset M$ be an open set and let $b, c\notin \cl I$ be distinct points. Then, there exists $M=M(\insideset,b,c)<\infty$ such that, for any $x,y,a\in I$ we have $(x,y|a,c)\le (x,y|b,c)+M.$ 
\end{lemma}
\begin{proof}
	Consider an open set $\outsideset\supset \insideset$ such that $\cl \insideset\subset \outsideset$ and $b,c\in O^c$. We can choose such a set, as $M$ is metrizable. By Lemma \ref{lemma:finiteamount} (taking $B_1=B_2=I,C_1=C_2=O^c$) we have
	\begin{equation}\label{eq:finite1}
	|\{A_j\in \mathcal{A} \mid \{a,c\}\subset A_j^-,A^+_j\cap \insideset\neq \emptyset, A^+_j\cap \outsideset^c\neq\emptyset \}|\le K.
	\end{equation}
	Assume $(a,c|x,y)=r\ge K+2$. Then we have a sequence
	$$\{a,c\}<A_1<\ldots <A_r <\{x,y\}.$$
	We will show now that $A^+_{K+1}$ is contained in $\outsideset$.
	To see this, recall that the definition of the relation between annuli implies $A^+_i\supset A^+_{i+1}$. Hence, if $A^+_i$ is contained in $\outsideset$ for some $i\le K$ we are done.
	Assume $A^+_i$ is not contained in $\outsideset$ for any $i\le K$, that is, $A_i^+\cap \outsideset^c\neq \emptyset$ for all $i\le K$.
	We have, not just for $i\le K$ but for all $i\le r$, that $A^+_i$ intersects $\insideset$, since $x\in \insideset$, and that $A^-_i$ contains $a$ and $c$.
	Hence, by \eqref{eq:finite1}, $A^+_{K+1}$ can not intersect $\outsideset^c$.
	By definition of the relation, $A_{K+2}^-$ contains $\outsideset^c$, so we have the chain
	$$\{b,c\}<A_{K+2}<\ldots <A_r <\{x,y\}.$$
	Then, since $(x,y|b,c)$ is the length of the maximal chain between $\{x,y\}$ and $\{b,c\}$, we have $(x,y|b,c)\ge r-(K+2)$ or, reorganizing, $(x,y|a,c)\le (x,y|b,c)+K+2$ and $M=K+2$.
\end{proof}
In particular, if $x_n$ and $y_n$ converge to some $x$, and $b,c$ are different from $x$, then there is some $n_0$ such that, for $n\ge n_0$, we have that $x_n$ and $y_n$ are in some open set as in the lemma, so we have $(x_n,y_n|a,c)\le (x_n,y_n|b,c)+K$.

Finally, the last application of the previous results is as follows.
\begin{lemma}\label{lemma:crossratiogromov}
	Let $I$, $I'$ be open subsets of $M$ with disjoint closures and let
	$a\in I$,  $b \notin \cl I\cup I'$ and $c\in I'$. Then there exists $K=K(I,I',a,b,c)<\infty$ such that, for all $w,x\in I$ and $y,z\in I'$,
	\[
	(w,x|y,z)\ge (w,x|b,c)+(a,b|y,z)-K.
	\]
\end{lemma}
\begin{proof}
	Let $O,O'$ be open sets with $\cl I \subset O$, $\cl I'\subset O'$, $\cl O \cap \cl O'= \emptyset$ and $b$ is not contained in the closure of any. Since $M$ is metrizable, we can choose such sets.
	Let $(w,x|b,c)=r$. By definition of the crossratio, we have the sequence
	$\{w,x\}<A_1^1<\ldots <A_r^1<\{b,c\}.$
	Hence, doing the same reasoning as before 
	$$\{w,x\}<A_1^1<\ldots <A_{r-K_1-2}^1<\outsideset^c,$$
	where $K_1$ is the constant obtained by applying Lemma \ref{lemma:finiteamount}.
	If $(a,b|y,z)=s$, for the sequence $\{a,b\}<A^2_1<\ldots <A^2_1<\{y,z\},$
	we will have 
	$$(\outsideset')^c<A_{K_2+2}^2<\ldots <A_s^2<\{y,z\},$$
	where $K_2$ is once again the constant obtained by applying Lemma \ref{lemma:finiteamount}.
	Therefore, since $\outsideset\subset (\outsideset')^c$ and $\outsideset^c\supset \outsideset'$, we can concatenate both sequences and we get
	$$\{w,x\}<A_1^1<\ldots <A_{r-K_1-2}^1<A^2_{K_2+2}<\ldots <A^2_s<\{y,z\}.$$
	Hence, the lemma is satisfied with $K(I,I',a,b,c)= K_1+K_2+4$.	
\end{proof}

\subsubsection{One direction, from the Gromov boundary $\partial T$ to $M$}
Given $\lambda\in \partial T$ take any sequence $(a_n)=((a^1_n,a^2_n,a^3_n))$ such that $(a_n)\sim \lambda$.
Since each component of the sequence is in the compact space $M$ we can take a subsequence, relabeled $a_n$, such that $a^1_n\to a^1$, $a^2_n\to a^2$ and $a^3_n\to a^3$ as $n\to \infty$.
If $(a^1,a^2,a^3)\in T$ then we can apply Lemma \ref{lemma:finitedistance} to see that every possible combination of $\rho(x,a_n)$ is bounded in $n$, and hence $a_n$ does not converge to a point in the Gromov boundary.
Therefore, at least two components converge to the same point $p$. We will define the function $\phi(\lambda):=p$, so we have to show that $p$ does not depend on the sequence $(a_n)$ (or the convergent subsequence).

\begin{lemma}\label{lemma:welldefined}
	Given $\lambda\in \partial T$ and a representing sequence $(a_n)_{n\in \N}=((a_n^1,a_n^2,a_n^3))_{n\in\N}\subset T$ such that two elements converge to $p$, then every other $(b_n)\subset T$ converging to $\lambda$ with Gromov topology converges to $p$ with Tukia's topology. In particular, we can define the function $\phi:\partial T \to M$ by $\phi(\lambda)=p$.
\end{lemma}
\begin{proof}
	Assume the lemma is false, i.e., that there exists $(b_n)=(b^1_n,b^2_n,b^3_n)$ with $(a_n\cdot b_n)_{x}\to \infty$ for some $x\in T$, but $(b_n)$ does not converge to $p$ in Tukia's topology.
	By definition of the topology, given an open set $V$ around $p$ there exists a subsequence of $(b_n)$ (which we relabel as $(b_n)$) such that there are always two components outside that open set.
	Hence, by compactness, we can take again a converging subsequence of $(b_n)$ (relabelled again $(b_n)$) such that two components converge to points outside $V$ and, since $\rho(x,b_n)\to \infty$, we can assume they converge to the same point $p'\neq p$. Fixing some $t\in M$ different from $p$ and $p'$, and denoting $x=(p,t,p')$, we compute the Gromov product
	\[
	2(a_n\cdot b_n)_{x}=\rho(a_n,x)+\rho(x,b_n)-\rho(a_n,b_n).
	\]
	Assume, reordering the components if necessary, that the first two components of $a_n$ and $b_n$ are the ones that converge to $p$ and $p'$ respectively.
	If the third component of $(a_n)$ converges to $\alpha\neq p$, then by Lemma \ref{lemma:finitedistance} we get that
	\[
	\max\left\{(a^i_n,a^3_n|v,w):v,w\in \{p,t,p'\}, i\in\{1,2\}\right\}<C.
	\]
	Since $\rho(a_n,x)\to \infty$, we have $\rho(a_n,x)=\max\left\{(a^1_n,a^2_n|v,w):v,w\in \{p,t,p'\}\right\}$ for $n$ greater than some $n_0$.
	If $a^3_n\to p$, we can reorder the components of each $a_n$ such that the distance to $x$ is always $\max\left\{(a^1_n,a^2_n|v,w):v,w\in \{p,t,p'\}\right\}$.

	Applying Lemma \ref{lemma:boundedcrossratiobychangingpoints}, we get $K_1<\infty$ and $n_0$ such that, for $n\ge n_0$, we have $(a^1_n,a^2_n|p,p')\le (a^1_n,a^2_n|t,p')+K_1$ and $(a^1_n,a^2_n|p,t)\le (a^1_n,a^2_n|t,p')+K_1$.
	Hence, for $n$ big enough, $\rho(a_n,x)\le (a_n^1,a_n^2|t,p')+K_1$.
	Doing the same reasoning for $b_n$ we get, for $n$ big enough,
	$$2(a_n\cdot b_n)_{x}\le (a^1_n,a^2_n|t,p')+(p,t|b^1_n,b^2_n)-\rho(a_n,b_n)+K_1+K_2.$$
	
	To bound from below the remaining term we use that $\rho(a_n,b_n)\ge (a^1_n,a^2_n|b^1_n,b^2_n)$. There exists some $n_0$ such that for $n\ge n_0$ the sets $\{a^1_n,a^2_n\}$ and $\{b^1_n,b^2_n\}$ are inside some sets $I$, $I'$ as described in Lemma \ref{lemma:crossratiogromov}, with $p,t'$ and $p'$ acting as $a,b,c$. Hence, applying the aforementioned lemma we get 
	\[(a^1_n,a^2_n|b^1_n,b^2_n)\ge (a^1_n,a^2_n|t,p')+(p,t|b^1_n,b^2_n)-K_3,
	\]
	from which it follows that
	$$2(a_n\cdot b_n)_{x}\le K_1+K_2+K_3.$$
	Since the bound does not depend on $n$ we have, by the reverse triangle inequality, $(a_n \cdot b_n)_{x}\ge \min((a_n\cdot \lambda)_x,(\lambda\cdot b_n)_x)+O(r)$. 
	Since $(a_n\cdot \lambda)_x$ goes to infinity, $(\lambda\cdot b_n)_x\le (a_n\cdot b_n)_x+O(r)\le K+O(r),$ and hence $b_n$ does not converge to $\lambda$ in the Gromov topology.
\end{proof}

As $\phi$ has been defined by using the convergence of the sequence we can get the following.

\begin{lemma}\label{lemma:equivandcont}
	The map $\phi$ is $G$-equivariant and continuous.
\end{lemma}
\begin{proof}
	Since $G$ respects the convergences to the boundaries on $T\cup \partial T$ and $T\cup M$, we get that $\phi$ is $G$-equivariant. 
	That is, if $x_n\to \lambda\in\partial T$ then $gx_n\to g\lambda \in \partial T$ and if $x_n\to p\in M$, $gx_n\to gp \in M$, so $\phi(g\lambda)=g\phi(\lambda)$.
	
	To see the continuity we take $\lambda_n\to \lambda$ and assume $\phi(\lambda_n)$ does not converge to $p:=\phi(\lambda)$.
	Since $(\phi(\lambda_n))$ is contained in a compact set we can take a convergent subsequence converging to some $p'\neq p$, which we relabel as $(\phi(\lambda_n))$.
	Consider sequences $(a^n_m)_m$ associated to each $\lambda_n$, and a sequence $(b_m)_m$ associated to $\lambda$. 
	The same reasoning as in in the last part of the proof of Lemma \ref{lemma:welldefined} can be repeated for each pair of sequences $(a^n_m)_m$ and $(b_m)_m$. 
	Since $\phi(\lambda_n)$ converges to $p'\neq\phi(\lambda)$, we can take $\insideset$ and $\insideset'$ to be fixed neighbourhoods of the points $p$ and $p'$ respectively with disjoint closures, and such that the chosen point $t$ is not contained in the closures. 
	There exists $n_0$, such that for $n\ge n_0$, $m\ge m(n)$  we have $a^n_m\in I$ and $b_m\in I'$.
	Then, doing the same reasoning as in the last proof, we get the bound 
	$$(a^n_m \cdot b_m)_{(p,t,p)}\le K(I,I',p,t,p'),$$
	which contradicts the hypothesis $\lambda_n\to \lambda$.
\end{proof}

\subsubsection{Finite boundary points}

As at the beginning of the section, we would like to be able to do the same for going from $M$ to $\partial T$.
That is, given $p\in M$, take any sequence $(x_n)\subset T$ with $x_n\to p$ and see that, in $T\cup \partial T$, $x_n\to\lambda\in \partial T$.
However, there may be problematic points for which $\rho(x_0,x_n)$ can be bounded.
We give the following definition.
\begin{definition}
	Let $p\in M$ and let $\rho$ be a quasimetric on $T$ defined using Sun's construction.
	We say that $p$ is a \emph{finite boundary} point if there exists a sequence $(x_n)\subset T$ and a number $R<\infty$ such that $x_n\to p$ in Tukia's topology, and $\rho(x_0,x_n)\le R$ for every $n$.
\end{definition}
It is immediate to check, using the triangle inequality, that the notion of finite boundary point does not depend on the basepoint $x_0$. 
It might, however, depend on the choice of quasimetric, that is, on the annulus we chose to define it.
The set of all finite boundary points will be called \emph{finite boundary} and will be denoted $\fsp$.
If a point is not a finite boundary point, we will call it \emph{infinite boundary} point, and the set of all infinite boundary points will be denoted $\nsp$($=\fsp^c$).

To see that $\fsp$ is not empty in some cases we recall a classical definition in the context of Kleinian groups, which can be easily generalized for convergence groups.
\begin{definition}
	Let $p\in M$. We say that $p$ is a \emph{conical limit} point if there exist $a,b\in M$ distinct, and a sequence $(g_n)\subset G$ such that $g_np$ converges to $a$ but $g_nx$ converges to $b$ for all $x\neq p$.
\end{definition}
On the context of uniformly convergence group Tukia shows in \cite{Tukia} that parabolic fixed points (that is, points fixed by some parabolic element) are exactly the non conical limit points.
We will show that non conical limit points are finite boundary points, and hence that $\fsp$ can be non empty.

First, we introduce a sufficient condition for being a finite boundary point. In Lemma \ref{lemma:boundp} we will see that it is actually an equivalent condition.

\begin{lemma}\label{lemma:noboundp}
	Let $p\in M$. If there are two points $a,b\in M\setminus p$ and 
	$R>0$ such that $(a,b|p)< R$, 
	then $p\in \fsp$.
\end{lemma}
\begin{proof}
	Fix $x=(a,b,p)\in T$.
	We need to find $(x_n)\subset T$ such that $x_n\to p$ and $\rho(x,x_n)\le K$. 
	We take as candidate $x_n=(a,t_n,p)$ where $t_n\to p$.
	By definition, the value of $\rho(x,x_n)$ is the maximum between $(a,b|t_n,p)$ and $(b,p|a,t_n)$. By hypothesis, $(a,b|t_n,p)$ is bounded, as any chain between $\{a,b\}$ and $\{t_n,p\}$ is also a chain between $\{a,b\}$ and $\{p\}$.
	Hence the only possibility for $p$ to be in the infinite boundary is $(b,p|a,t_n)\to \infty$, which can not happen by Lemma \ref{lemma:finitedistance} applied with $B_1=\{b\},C_1=\{p\}, C_2=\{a\}$ and $B_2$ an open set around $p$ separated from $b$.
	
\end{proof}

\begin{proposition}\label{prop:conicpointsareinfinitebdpoints}
	Let $p\in \nsp$. Then, there exists a sequence $(g_k)\subset G$ such that $g_k p\to c\in A^\pm$, and $g_k x\to b\in A^\mp$ for all $x\neq p$. In particular, $p$ is a conical limit point.
\end{proposition}
\begin{proof}
	Since $p\in \nsp$, the quantity $(a,b|p)$ is unbounded for any $a,b\in M\setminus p$.
	Hence, we can build arbitrary long chains of the form
	$$\{a,b\}<A_1<\ldots <A_{2R}<\{p\}.$$
	Then we can take subchains such that
	$$\{a,b\}<g_1\sigma(R)A<\ldots <g_R\sigma(R)A<\{p\},$$
	where $\sigma(R) =\pm 1$.
	That is, we can take a subchain such that all annulus of the chain are either translates of $A$ or of $-A$.
	Assume that as $R\to \infty$, $\sigma(R)=1$ infinitely many times. 
	Then, we have infinitely many $h\in G$ such that $\{a,b\}\subset hA^-$ and $p\in hA^+$. 
	Taking a convergent subsequence $(h_i)$, by the reasoning of Lemma \ref{lemma:finiteamount}, the repelling point is in $A^-$.
	Also, $p\in h_iA^+$, so $h_i^{-1}p\in A^+$. 
	Since $M$ is compact we can take a subsequence such that $h_{i_k}^{-1}p$ converges and, since $A^+$ is closed, it converges to a point in $A^+$. 
	Then, the sequence $(g_k)=(h_{i_k}^{-1})$ has its attracting point in $A^-$, and $g_k p$ converges to a point in $A^+$, so the proposition is satisfied.
	If instead we have $\sigma(R)=-1$ infinitely many times, following the same reasoning we get $g_kp\to c\in A^-$ and $g_kx\to b\in A^+$ for all $x\neq p$.
\end{proof}
So, in particular, parabolic fixed points are finite boundary points and $\fsp$ is not empty in some cases. 

\subsubsection{Inverse, from $\nsp$ to $\phi^{-1}(\nsp)$}
We have defined a continuous $G$-invariant map $\phi$ from $\partial T$ to $M$. 
This has been built by observing that, given $\lambda \in \partial T$, any sequence converging to $\lambda$ in Gromov's topology converges to a fixed point $\phi(\lambda)\in M$ in Tukia's topology. 
As we have seen, the same reasoning can not be used to build the inverse, as some sequences converging to finite boundary points may be bounded and hence not converge to $\partial T$ in Gromov's topology.
To get around this we simply forget about the problematic points and build the inverse from $\nsp$ to $\phi^{-1}(\nsp)$. 
Later we will see that, under the stationary probability measure, the mass of the set we are leaving out is $0$ so we have an equivalence between $G$-measure spaces.

Given $p\in \nsp$, we will associate a really particular kind of sequence $(x_n)\subset T$ which converges to $p$ in Tukia's topology and prove that, in the Gromov topology, $(x_n)$ converges to some $\lambda\in \partial T$ only depending on $p$.
The associated sequence we choose is based on a construction found in \cite{Tukia}. 
Given a bi-infinite geodesic on the Kleinian space between two boundary points $a,b \in \S^n$ one considers a subset of the preimage to $T$, $L(a,b)=\{(a,b,t),t\in M\setminus \{a,b\}\}$, denoted \emph{line} (between $a$ and $b$). 
On the Kleinian example we have that a line projects to a whole bi-infinite geodesic.
For a general convergence group we expect that lines have some similarities with quasigeodesics in $(T,\rho)$.
In particular, if $a$ and $b$ are infinite boundary points, the Gromov boundary of $L(a,b)\subset T$, $\partial L(a,b)\subset \partial T$, consists of two points.
Furthermore, if we fix $a$ and move $b$, one of the two Gromov boundary points does not change, effectively providing a well-defined point in the Gromov boundary associated to $a$.

With this in mind, given $p\in \nsp$, we consider any $\alpha \neq p$ and a sequence $t_n\to p$. 
Since $p$ is not a finite boundary point, the sequence $(\alpha,t_n,p)\subset T$ goes to infinity, as $(\alpha,t_n,p)\to p$. 
If our reasoning from before is correct, this defines a Gromov sequence (or at least contain a Gromov subsequence) converging to some point in $\partial T$, which we denote $\psi(p)$. 
All the claims we made in the previous paragraph can be deduced from the following lemma.
\begin{lemma}\label{lemma:welldefined1}
	The application $\psi:M_\infty \to \partial T$ is well-defined.
\end{lemma}
\begin{proof}
	Fix $x=(\alpha,t_0,p)$ and choose $(t_n)\subset M$ with $t_n\to p$.
	and denote $x_n=(\alpha,t_n,p)$.
	Since $p\in M_\infty$ and $x_n\to p$ we have $\rho(x,x_n)\to \infty$.
	By definition we have $\rho(x,x_n)=\max((p,t_0|t_n,\alpha),(\alpha,t_0|t_n,p))$.
	Since $t_n\to p\neq \alpha$ we can apply Lemma \ref{lemma:finitedistance} to get $(p,t_0|t_n,\alpha)<K$.
	Hence,
	\[
	\rho(x,x_n)=(\alpha,t_0|t_n,p)
	\]
	for $n$ large enough. 
	Then, the Gromov product of the sequence $(x_n)$ results in, for $n,m$ large enough,
	\[
	2(x_n\cdot x_m)_{x}=(\alpha,t_0|t_n,p)+(\alpha,t_0|t_m,p)-\max((\alpha,t_n|t_m,p),(p,t_n|t_m,\alpha)).
	\]
	Take a neighborhood $I$ around $p$, separated from $\alpha$ and $t_0$. There is some $n_0$ such that for $n,m\ge n_0$ we have $t_n,t_m\in I$, so applying Lemma \ref{lemma:boundedcrossratiobychangingpoints} we have some $K<\infty$ such that, for any $n,m\ge n_0$,
	\[(\alpha,t_n|t_m,p)\le (\alpha, t_0|t_m,p) + K \text{ and } (\alpha,t_m|t_n,p)\le (\alpha, t_0|t_n,p) + K.
	\]
	Therefore, $\rho (x_n,x_m)\le \max ( (\alpha, t_0|t_n,p), (\alpha, t_0|t_m,p)) + K$, and 
	$$2(x_n\cdot x_m)_{x}\ge \min(\rho(x,x_n),\rho(x,x_m))-K,$$
	which goes to infinity as $n,m$ go to infinity. So, $(x_n)$ converges to some $\lambda\in \partial T$ in Gromov's topology, which we define as $\psi(p)$.
	The same reasoning can be applied to see that any other sequence with $t_m'\to p$ satisfies $(x_n\cdot (\alpha,t_m',p))_{x_0}\to \infty$, so the corresponding sequence converges to the same $\lambda$.
	
	The only remaining thing to check to see that $\psi$ is well-defined is that $\lambda$ does not depend on $\alpha$. 
	For this, we see that a different $\alpha$ displaces the tail of the sequence by a bounded distance.
	That is, we see that for $n$ big enough,
	$$\rho((\alpha,t_n,p),(\beta,t_n,p))=\max((\alpha,t_n|\beta,p),(\alpha,p|\beta,t_n))<C.$$
	This follows easily from applying Lemma \ref{lemma:finitedistance}. %choosing the neighbourhood of $p$, $\outsideset$, such that $\cl \outsideset$ does not contain $\alpha$ nor $\beta$.
\end{proof}

The next thing we want to see is that $\psi$ is actually the inverse of the previous application $\phi$. 
That is, all Gromov sequences converging to $p\in \nsp$ are actually similar and hence related to the same $\lambda\in \partial T$.
\begin{lemma}\label{lemma:phihasinverse}
	The restriction of the application $\phi:\partial T \to M$ to $\phi:\phi^{-1}(M_\infty)\to M_\infty$ has an inverse, given by the function $\psi$ described above.
\end{lemma}
\begin{proof}
	Since $\psi(p)=((\alpha,t_n,p))_n$ with $t_n\to p$, we have $\phi(\psi(p))=p$.
	
	We have to check $\psi(\phi(\lambda))=\lambda$, i.e., that given $\lambda\in \phi^{-1}(p)$ where $p$ is not a finite boundary point and a sequence $(a_n)=((x_n,y_n,z_n))_n \sim \lambda$, we have $((x_n,y_n,z_n))_n\sim((\alpha,t_n,p))_n$.
	
	The first step will be to show that $((x_n,y_n,z_n))_n\sim((\alpha,y_n,z_n))_n$.
	We begin by fixing $x=(\alpha,t,p)$ and taking a subsequence of $(a_n)$ such that each of the three elements converges to some point. 
	By Lemma \ref{lemma:welldefined1} at least two of these have to converge to $p$, which we assume are the last two components. 
	By Lemma \ref{lemma:finitedistance}, if the first component converges to $\beta\neq p$, the tails of the sequences $(x_n,y_n,z_n)$ and $(\alpha,y_n,z_n)$ are separated by a bounded distance, so we have the similarity between the sequences are similar. 
	If $x_n\to p$, shuffling the components if necessary, we may assume that the distance from $x$ is always achieved with the last two components (i.e., $\rho(x,a_n)=\max\left\{(v,w|y_n,z_n):v,w\in \{\alpha,t,p\}\right\}$). 
	Evaluating the Gromov product, we get
	$$2(a_n\cdot (\alpha,y_n,z_n))_{x}=\rho(x,a_n)+\rho(x,(\alpha,y_n,z_n))-\rho((x_n,y_n,z_n),(\alpha,y_n,z_n)).$$
	The last term is equal to $\max((x_n,y_n|\alpha,z_n),(x_n,z_n|\alpha,y_n))$, so, applying Lemma \ref{lemma:boundedcrossratiobychangingpoints} we have some $K<\infty $ such that
	\[
	(x_n,y_n|\alpha,z_n)\le (x_n,y_n|\alpha,t) +K\text{ and } (x_n,z_n|\alpha,y_n)\le (x_n,z_n|\alpha,t) +K,
	\]
	which are both smaller than $\rho(x,a_n) + K$. Hence,
	$$2(a_n\cdot(\alpha,y_n,z_n))_{x}\ge \rho(x,(\alpha,y_n,z_n))-K.$$
	Since we assumed that the distance between $x$ and $a_n$ is achieved with the last two components of $a_n$ we have $\rho(x,(\alpha,y_n,z_n))\ge \rho(x,a_n)$, so the Gromov product goes to infinity and both sequences are similar.
	
	The next step is checking that $((\alpha,y_n,z_n))_n\sim((\alpha,y_n,p))_n$. The Gromov product is
	$$2((\alpha,y_n,z_n)\cdot(\alpha,y_n,p))_{x}=\rho(x,(\alpha,y_n,z_n))+\rho(x,(\alpha,y_n,p))-\rho((\alpha,y_n,z_n),(\alpha,y_n,p)).$$
	By the definition of the distance, the first term satisfies
	\begin{equation*}
	\rho(x,(\alpha,y_n,z_n))\ge (y_n,z_n|\alpha,p).
	\end{equation*}
	The second term is $\rho((\alpha,t,p),(\alpha,y_n,p))=\max((\alpha,t|y_n,p),(t,p|\alpha,y_n))$. 
	By Lemma \ref{lemma:finitedistance} we have $(t,p|\alpha,y_n)<K$. 
	Since $p$ is in the infinite boundary, $(\alpha,t|y_n,p)\to \infty$ as $y_n\to p$.
	By Lemma \ref{lemma:boundedcrossratiobychangingpoints}, for $n$ big enough 
	\begin{equation*}
	\rho((\alpha,t,p),(\alpha,y_n,p))=(\alpha,t|y_n,p)\ge (\alpha,z_n|y_n,p)-C.
	\end{equation*}
	Using these last two inequalities we get
	\begin{align*}
	\rho((\alpha,y_n,z_n),(\alpha,y_n,p))&=\max((\alpha,z_n|y_n,p),(y_n,z_n|\alpha,p))\\
	&\le \max (\rho(x,(\alpha,y_n,p)),\rho(x,(\alpha,y_n,z_n)))+C.
	\end{align*}
	Then, we have
	$$2((\alpha,y_n,z_n)\cdot (\alpha,y_n,p))_{x}\ge \min (\rho(x,(\alpha,y_n,p)),\rho(x,(\alpha,y_n,z_n)))-C,$$
	which goes to infinity since the $p$ is an infinite boundary point (and the first possible value goes to infinity) and we had chosen $y_n$ and $z_n$ such that $\rho(x,(\alpha,y_n,z_n))\ge \rho(x,a_n)\to \infty$.
	
	Finally, by the proof of Lemma \ref{lemma:welldefined1}, for $p$ in the infinite boundary all sequences of the form $(\alpha,t_n,p)$ with $t_n\to p$ are equivalent so we have $(a_n)\sim ((\alpha,t_n,p))$ and hence $\psi(\phi(\lambda))=\lambda$.
\end{proof}

The only thing remaining to get a homeomorphism is the continuity of $\psi$. 
\begin{lemma}\label{lemma:psicontinuous}
	The application $\psi:\nsp\to \partial T$ is continuous.
\end{lemma}
\begin{proof}
	Consider $(p_m)\subset \nsp$ converging to $p\in\nsp$. We want to show that $(\lambda_m)=(\psi(p_m))$ converges to $\lambda=\psi(p)$. Fix $x=(\alpha,t,p)$ and let $t_n\to p$. For each $n$, we have the maximal sequence 
	\[\{\alpha,t\}<A_{1,n}<\ldots <A_{k(n),n}<\{t_n,p\}
	\]
	associated to $(\alpha,t|t_n,p)$.
	For the last term of the sequence we have $p\in \inte A^+_{k(n),n}$, so we can take $m(n)$ such that $p_{m}$ is inside $\inte A^+_{k(n),n}$ for all $m\ge m(n)$. 
	Then, for all $y\in \inte A^+_{k(n),n}$ we have \[\{\alpha,t\}<A_{1,n}<\ldots <A_{k(n),n}<\{y,p_m\}.
	\]
	Hence, $\rho(x,(\alpha,y,p_m))$ and $\rho(x,(\alpha,y,p))$ are both greater than $(\alpha,t|t_n,p)$.
	Evaluating the Gromov product, we obtain
	$$2((\alpha,y,p_m)\cdot(\alpha,y,p))_{x}= \rho(x,(\alpha,y,p_m))+\rho(x,(\alpha,y,p))-\rho((\alpha,y,p_m),(\alpha,y,p)).$$
	By definition of the distance, the first term satisfies
	\[ (y,p_m|\alpha,p)\le \rho(x,(\alpha,y,p_m)).
	\]
	The second term is $\rho((\alpha,t,p),(\alpha,y,p))=\max((\alpha,t|y,p),(t,p|\alpha,y))$. By Lemma \ref{lemma:finitedistance} we have $(t,p|\alpha,y)<K$, so we can take $n$ big enough (modifying $p_m$ accordingly) for which we have $\rho((\alpha,t,p),(\alpha,y,p))=(\alpha,t|y,p)$, which is greater than $(\alpha,p_m|y,p)-C$, that is, such that
	\[(\alpha,p_m|y,p)\le \rho(x,(\alpha,y,p))+C. \]
	Using these last two inequalities, we get
	\begin{align*}
	\rho((\alpha,y,p_m),(\alpha,y,p))&=\max((\alpha,p_m|y,p),(y,p_m|\alpha,p))\\
	&\le \max (\rho(x,(\alpha,y,p)),\rho(x,(\alpha,y,p_m)))+C.
	\end{align*}
	%	As before, the last term is 
	%	$$\rho((\alpha,y,p_m),(\alpha,y,p))=\max((\alpha,p_m|y,p),(y,p_m|\alpha,p)).$$
	Hence,
	\begin{align*}
	2((\alpha,y,p_m)\cdot(\alpha,y,p))_{x}&\ge \min (\rho(x,(\alpha,y,p_m)),\rho(x,(\alpha,y,p)))-C\\
	&\ge \rho(x,(\alpha,t_n,p)))-C.
	\end{align*}
	
	Given $L>0$, consider a neighborhood $V$ of $p$ such that $((\alpha,t_i,p)\cdot (\alpha,t_j,p))_{x}\ge L$ for all $t_i,t_j\in V$. If there is no such neighborhood we get a contradiction with the proof of Lemma \ref{lemma:welldefined1}, since we could make a sequence $(t_n)$ converging to $p$ where $(\alpha,t_n,p)$ does not converge to $\psi(p)$. Take any $t_j\in V$, let $U:=\inte A^+_{k(j),j}$ and let $m(j)$ be as in the beginning of the proof. Take $m\ge m(j)$ big enough so $p_m\in V\cap U$.	For all points $t^m,t^p\in U\cap V$ we have 
	$$((\alpha,t^m,p_m)\cdot(\alpha,t^m,p))_{x}\ge L-C$$
	and
	$$((\alpha,t^m,p)\cdot(\alpha,t^p,p))_{x}\ge L.$$
	So, by the triangle inequality for the Gromov product,
	$$((\alpha,t^m,p_m)\cdot(\alpha,t^p,p))_{x}\ge L-C-O(r).$$
	Therefore, for any $L$ there exists $m_0$ such that for $m\ge m_0$, if $t^m_n\to p_m$ and $t^p_n\to p$,
	the Gromov products between elements of the tails of the associated Gromov sequences $(\alpha, t^m_n,p_m)$ and $(\alpha,t^p_n,p)$ is bigger than $K-C-O(r)$. Hence, given $K$ there exists $m$ such that for $m\ge m_0$, $(\lambda_m\cdot \lambda)_{x}\ge K$, so $\lambda_m\to \lambda$.
\end{proof}

Joining Lemmas \ref{lemma:welldefined}, \ref{lemma:equivandcont}, \ref{lemma:welldefined1}, \ref{lemma:phihasinverse} and \ref{lemma:psicontinuous}, we have the following.
\begin{proposition}\label{prop:phihomeo}
	Let $G$ be a minimal, non-elementary convergence group on compact a metric space $M$. Let $(T,\rho)$ be the set of distinct triples equipped with the quasimetric described by Sun. Then, there exists a $G$-equivariant continuous map $\phi:\partial T\to M$ such that the restriction to $\phi^{-1}(M_\infty)$ is a homeomorphism between $\phi^{-1}(M_\infty)$ and  $M_\infty$.
\end{proposition}

\subsection{Zero sets of $M$ under the stationary measure}\label{proof}

In this section we will show that, under the stationary measure, the finite boundary has zero mass. The main ingredient we use is the following lemma, found in \cite[Lemma 4.5]{Tiozzo2}.
\begin{lemma}\label{lemma:zeromeasure}
	Let $G$ a countable group acting by homeomorphisms on a
	metric space $M$, $\mu$ a probability distribution on $G$ whose support generates $G$, and $\nu$ a $\mu$-stationary probability measure on $M$. Moreover,
	let us suppose that $Y\subset M$ has the property that there is a sequence of
	positive numbers $(\epsilon_n)$ such that for any translate $fY$ of $Y$ there is a
	sequence $(g_n)$ of group elements (which may depend on $f$), such that the
	translates $f Y$, $g_1^{-1} f Y$, $g_2^{-1} f Y$, $\ldots$ are all disjoint, and for each $g_n$, there is an $m\in N$, such that $\mu_m (g_n )>\epsilon_n$. Then $\nu(Y )= 0.$
\end{lemma}

Given $x\in T$ and $R>0$, we will consider the sets of finite boundary points which are at ``distance'' smaller than or equal to $R$ from $x$, that is, the points $p\in M$ such that there exists a sequence $(x_n)\subset T$ with $\rho(x,x_n)\le R$ such that $x_n\to p$. In other words, we define
$$\BM(x,R):=\overline{B(x,R)}\cap M,$$
where the closure is with respect to Tukia's topology.
A critical observation is that, by the definition of $M_F$, if $(R_i)_{i\in \N}$ is an ascending sequence going to infinity we get $\BM(x,R_i)\subset \BM(x,R_{i+1})$, and $\bigcup_{i\in \N}\BM(x,R_i)=M_F$.
Hence, if we prove that each of $\BM(x,R)$ has zero measure, then the ascending limit $M_F$ also has zero measure.

The first step we need to take to apply the lemma is proving that these balls behave well under the action $G$.
\begin{lemma}\label{lemma:natural}
	We have $g\BM(x,R)=\BM(gx,R)$, or, equivalently, $p\in \BM(x,R) \iff gp \in \BM(gx,R).$
\end{lemma}
\begin{proof}
	The point $p$ belongs in $\BM(x_0,R)$ if and only if there exists a sequence $(x_n)\subset B(x,R)$ with $x_n\to p$. Since $G$ acts by isometries on $T$ and by homeomorphisms on $M$, the previous is equivalent to $(gx_n)\subset gB(x,R)=B(gx,R)$ and $gx_n\to gp$, which is equivalent to $gp$ belonging in $\BM(gx_0,R)$.
\end{proof}

Next, we need to see that if the centers of the balls are far enough with respect to the radius, then the balls are disjoint. To prove that we first need a small lemma, which will come up later.
\begin{lemma}\label{lemma:boundp}
	If $x=(x^1,x^2,x^3)$ and $p\in \BM(x,R)$, then $(x^i,x^j|p)\le R$ whenever $i\neq j$, $1\le i,j\le 3$.
\end{lemma}
\begin{proof}
	Assume $p\in \BM(x,R)$ and that the conclusion is false. 
	That is, assume that there exists $x_n\to p$ with $\rho(x,x_n)\le R$ and that there exists $i,j$ such that $(x^i,x^j|p)\ge R+1$. 
	Then, there exists an annulus sequence of length $R+1$ such that
	$$\{x^i,x^j\}<A_1<A_2<\ldots < A_{R+1}<\{p\}.$$
	By definition of the order relation, $p\in \inte A^+_{R+1}$, so $\inte A^+_{R+1}$ is a neighborhood of p. By definition of the convergence to the boundary, there exists $n_0$ big enough such that for all $n\ge n_0$ at least two components of $x_n$ are in $\inte(A^+_{R+1})$. Therefore, we also have the chain
	$$\{x^i,x^j\}<A_1<A_2<\ldots < A_{R+1}<\{x^k_n,x^l_n\},$$
	and hence $\rho(x,x_n)\ge R+1$, which is a contradiction.
\end{proof}

As a side note, coupling this last result with Lemma \ref{lemma:noboundp} we get an equivalent definition of finite boundary point.
\begin{corollary}
	Let $p\in M$. Then $p$ is a finite boundary point if and only if there are two points $a,b\in M\setminus p$ such that $(a,b|p)< \infty$.
\end{corollary}

Next we prove that if the centers of the balls are far enough, then the balls are disjoint.
%The idea is to assume that they intersect, and use the annulus sequence between the two centers to build two sequences between each of the centers and the intersecting point, such that the sum of the lengths of these sequences is almost the length of the total sequence, which will be a contradiction by the previous lemma (since then, at least one of them has to bee too large).
\begin{lemma}\label{lemma:disjoint}
	If $\rho(x,y)\ge R+ 2$, then $ \BM(x,R)\cap \BM(y,R)=\emptyset$.
\end{lemma}
\begin{proof}
	Write $x=(x^1,x^2,x^3)$, $y=(y^1,y^2,y^3)$ and assume $p\in \BM(x,R)\cap \BM(y,R)$. Assume that the distance between $x$ and $y$ is realized by $(x^1,x^2|y^1,y^2)$. Then, $(x^1,x^2|y^1,y^2)\ge 2R+2$, and hence we have the chain
	$$\{x^1,x^2\}<A_1<A_2<\ldots <A_{2R+2}<\{y^1,y^2\}.$$
	By definition of the relation, $\inte A^-_{i+1}\cup \inte  A^+_{i}=M$, so for each $i$, the point $p$ belongs to either $\inte A^-_{i+1}$ or $\inte A^+_i$.
	Let $i_0$ be the biggest $i$ such that $p\in \inte A^+_i$. We have the chain
	$$\{x^1,x^2\}<A_1<A_2<\ldots <A_{i_0}<\{p\},$$
	so if $i_0\ge R+1$ we get a contradiction with Lemma \ref{lemma:boundp}, since $(x^1,x^2|p)\ge R+1$. If $i_0\le R$ we have that $p\notin \inte A^+_{i_0+1}$, so, by definition of the relation, $p\in \inte  A^-_{i_0+2}$. We get the chain
	$$\{p\}<A_{i_0+2}<A_{i_0+3}<\ldots <A_{2R+2}<\{y^1,y^2\},$$
	and we have again a contradiction with Lemma \ref{lemma:boundp}.
\end{proof}

With this we can prove the result we anticipated at the beginning of the section.
\begin{proposition}\label{stickyzero}
	Let $G$ be a minimal, non-elementary convergence group on a compact metrizable space $M$, and $\mu$ a probability measure on $G$ such that its support generates $G$. If $\nu$ is the $\mu$-stationary Borel probability measure, then $\nu(\fsp)=0$.
	
\end{proposition}
\begin{proof}
	We want to apply Lemma \ref{lemma:zeromeasure} with $Y=\BM(x,R)$, where $x\in T$ is fixed. By Lemma \ref{lemma:natural}, all translations of $Y$ are of the form $\BM(y,R)$. Let $g\in G$ be the loxodromic element determined by Sun in \cite{binsun}. 
	By Proposition \ref{prop:loxodromic}, we can take $N>0$ such that $\inf(\rho(x,g^{nN}x))\ge n(2R+2)$. 
	By Lemma \ref{lemma:disjoint}, for any $f\in G$ the sets $fY, g^{-N}fY,g^{-2N}fY,\ldots$ are disjoint, since the distance between any of the centers of the balls is greater than or equal to $2R+2$. 
	Since the support of $\mu$ generates $G$, there is $m(n)$ such that $\mu_{m(n)}(g^{nN})>0$, so labeling $\epsilon_n:=\mu_{m(n)}(g^{nN})$ we can apply Lemma \ref{lemma:zeromeasure} and we get $\nu(\BM(x,R))=0$.
	We finish by recalling that $\BM(x,R)\to \fsp$ as $R\to \infty$.
\end{proof}

The set where the measure has all of its mass can be restricted a little further. To do this, we observe that $\nsp$, which has full mass, may depend on the metric $\rho$, which in turn only depends on the chosen annuli system. Since we always deal with annuli systems generated by a single annulus $A:=\{A^-,A^+\}$, the infinite boundary depends only on the annulus $A$, or more specifically, on the sets $A^-,A^+$, so we can write $\nsp(A^-,A^+)$. Therefore, if we choose a countable family of annuli $\{A^-_i,A^+_i\}_{i\in \N}$ such that Sun's construction works, we will get a countable family of sets, $\nsp^i:=\nsp(A^-_i,A^+_i)$, where $\nu(\nsp^i)=1$, and by intersecting them we still get $\nu(\bigcap_{i\in \N}\nsp^i)=1$. 

Looking at Sun's construction we see that, for the construction to work, the conditions on $A^-$ and $A^+$ are the following:
\begin{itemize}
	\item $A^-$ and $A^+$ are closed, and $A^-\cap A^+=\emptyset$;
	\item there exists an element $g\in G$ behaving like the one described in Theorem \ref{theo:preloxodromic}, such that if $a^-,a^+\in M$ are its fixed points, $a^-\in \inte A^-$ and $a^+\in \inte A^+$.
\end{itemize}

Choosing a particular family of acceptable generating annuli we get the following.
\begin{proposition}\label{prop:ultraconic}
	Let $\nu$ be the Borel $\mu$-stationary measure on $M$ and let $g\in G$ be such that it fixes two distinct points $a^-,a^+$ and that $g^n|_{M\setminus a^-}$ converges to $a^+$ locally uniformly as $n\to \infty$. 
	Denote $\nsp^g$ the set of points $p\in M$ such that there exists a sequence $(g_n)$ (depending on $p$) with $a^-$ and $a^+$ being its attracting and repelling points (they can be either).
	Then $\nu(\nsp^g)=1$.
\end{proposition}
\begin{proof}
	Equip $M$ with a metric $\d_M$, and define the sets 
	$$A^-_i=\overline{B(a^-,\d_M(a^-,a^+)/i)} \text{ and } A^+_i=\overline{B(a^+,\d(a^-,a^+)/i)}.$$
	For $i\ge 3$, these sets define admissible annuli, and we get a family of annuli systems as described above and an associated family of infinite boundary points $\nsp^i$. The countable intersection $\widetilde{\nsp^g}:=\bigcap_{i\ge 3} \nsp^i$ has full mass, so let's see how any $p\in \widetilde{\nsp^g}$ behaves. By Proposition \ref{prop:conicpointsareinfinitebdpoints}, for each $i$ we will have $(g_n^i)_n$ with either $g_n^ip\to A^-_i$ and $g_n^i x \to A^+_i$ for all $x\neq p$ or $g_n^ip\to A^+_i$ and $g_n^i x \to A^-_i$ for all $x\neq p$. We assume now that the first one happens infinitely many times for $(i_k)$, and we take a convergent subsequence of each $(g_n^{i_k})$ (which we relabel as $(g_n^{i_k})$). Taking $V$ an open set around $p$, for each $i_k$ there exists $n_k$ big enough so $g_{n_k}^{i_k}p\subset A^-_{i_k}$ and $g_{n_k}^{i_k} (M-V)\subset A^+_{i_k-1}$. 
	By definition of the sets, as $k\to \infty$ we have $g_{n_k}^{i_k}p\to a^-$ and $g_{n_k}^{i_k}(M-V)\to a^+$. So, taking a convergent subsequence $(h_j)\subset(g_{n_k}^{i_k})_k$, we get a sequence with $a^+$ as the attracting point and $a^-$ as the repelling. Hence, $\widetilde{\nsp^g}\subset\nsp^g$, so $\nu(\nsp^g)=1$.
\end{proof}

\subsection{The Poisson boundary of convergence groups}

By Sun's construction \cite{binsun}, the action of $G$ on $(T,\rho)$ (or rather, on the quasi-isometric metric space $(S,\rho')$) has a WPD element and is non-elementary. So, using Maher and Tiozzo's theorem \cite[Theorem 1.4]{Tiozzo3}, if $\mu$ satisfies the required conditions, then the Gromov boundary of $S$ coupled with the hitting measure $\nu$ is a model for the Poisson boundary of $(G,\mu)$.
Using Propositions \ref{prop:phihomeo} and \ref{stickyzero}, we are able to prove Theorem \ref{theo:poissboundary}.

\begin{proof}[Proof of Theorem \ref{theo:poissboundary}]
	Let $(S,\rho')$ be the hyperbolic space quasi-isometric to $T$ by the $G$-equivariant quasi-isometry $f$ obtained in \cite{binsun}.
	To apply \ref{theo:poissonacylindrical} to $(S,\rho')$ we need to see that the measure $\mu$ has finite logarithmic moments. 
	For this we denote by $\d_w$ the word metric with reference to some finite generating set $H$. 
	By definition of the word metric, for each $g$ we have $h_1,h_2,\ldots h_{d_w(e,g)}\in H$ such that $g=h_{1} h_{2}\ldots h_{{\d_w(e,g)}}$. 
	Hence, as $\rho$ is a quasimetric, using the triangle inequality and the invariance of $\rho$, we obtain
	$$\rho(x_0,gx_0)\le \sum_{i=1}^{\d_w(e,g)} \rho(x_0,h_{i} x_0) + r \le (\sup_{h\in H}(\rho(x_0,h x_0))+r)d_w(e,g)= C d_w(e,g).$$
	Using this, we get a bound for the logarithmic moment of $\mu$ under the distance $\rho'$. 
	Looking at the definition of $\rho'$ in Sun's construction, we see that if $\rho(x,y)=0$, then $\rho'(x,y)=1$, and that $\rho'(x,y)\ge1$ whenever $x\neq y$. 
	Hence, 
	$$\E[\log(\rho'(x_0,gx_0))]= \E[\log(\rho'(x_0,gx_0));\rho(x_0,gx_0)\ge 1]$$
	Therefore, using the upper bound $\rho'(x_0,gx_0)) \le K\rho(x_0,gx_0)+K$, we get
	\begin{align*}
	\E[\log(\rho'(x_0,gx_0))]&\le \E[ \log( K\rho(x_0,gx_0)+K); \rho(x_0,gx_0)\ge 1]\\
	&\le \E[ \log( KC\d_w(e,g)+K)]\le\log(KC)+K+\E[\log(\d_w(e,g))],
	\end{align*}
	which is finite by hypothesis. Hence, we can apply Theorem \ref{theo:poissonacylindrical} and, denoting the hitting measure on $\partial S$ as $\tilde{\tilde{\nu}}$, we have that $(\partial S,\tilde{\tilde{\nu}})$ is the Poisson boundary of the random walk $(G,\mu)$.
	Recall that $\partial S$ and $\partial T$ are homeomorphic by the induced action of $f$, which is $G$-equivariant, so $\tilde{\nu}:=f_*(\tilde{\tilde{\nu}})$ is $\mu$-stationary, and $(\partial T,\tilde{\nu})$ is equivalent to $(\partial S,\tilde{\tilde{\nu}})$ as measure space (and hence it is the Poisson boundary).
	
	Let $\phi$ be the $G$-equivariant application built in section \ref{sect:gromboundrayofT}, and $\psi$ the inverse on $M_\infty$. By $G$-equivariance, the probability measure $\nu:=\phi_*(\tilde{\nu})$ on $M$ is also $\mu$-stationary, and by continuity it is Borel so, by Proposition \ref{stickyzero}, $\nu(\nsp)=1$. Therefore, $\psi_*(\nu)=\tilde{\nu}$ so the two spaces are equivalent as measure spaces via a $G$-equivariant map, and hence $(M,\nu)$ is the Poisson boundary.
\end{proof}

\section{Applications}

\subsection{Compactification of $G$}\label{sect:sectcompact}

Using the topology we used to paste $M$ to $T$ we can paste $M$ to $G$ in a similar way.
That is, fix $x\in T$ and for any $U\subset M$ open we can consider the subset of $G\cup M$ given by 
$$\widetilde{U}_G:=\{g\in G\mid gx\text{ has two components in } U\}\cup U.$$
The family $B$ of sets of this form, together with $\mathcal{P}(G)$ (that is, all open sets of $G$, as it has the discrete topology), forms a basis for a topology on $G\cup M$. Therefore, we may consider the generated topology.

\begin{proposition}
	The topology on $G\cup M$ defined above does not depend on the basepoint $x$.
\end{proposition}
\begin{proof}
	Given a point in $M$, we can take a countable neighborhood basis in $M$, and we get a corresponding countable neighborhood basis in $G\cup M$. Since $G$ has the discrete topology, a point in $G$ itself is a neighborhood, so we have a countable neighborhood basis. Therefore, $G\cup M$ is first countable, and the topology is characterized by convergence along sequences. 
	
	Consider a sequence $(g_n)\subset G$ with $g_n\to p$, which by definition is equivalent to $g_nx\to p$ in Tukia's topology.
	Take $y\in T$ and assume that $g_{n}y$ does not converge to $p$.
	Then there exists an open neighborhood of $p$ in $T\cup M$ of the form $\widetilde{U}$ and a subsequence of $(g_{n_k})\subset (g_n)$ such that $g_{n_k}y$ does not enter $\widetilde {U}$.
	However, by the convergence property, we can take a convergent subsequence of $g_{n_k}$ (in the sense of convergence groups) which, since $g_nx\to p$, has $p$ as attracting point.
	Hence $g_{n_k}y$ enters $ \widetilde{U}$ eventually and hence $g_{n_k}$ enters $\widetilde{U}_G$. 
	That is, $g_n$ converges to $p$ in the topology generated by taking $y$ as a basepoint, so doing the same reasoning the other way, $g_n$ converges to $p$ in the topology generated by taking $y$ as basepoint if and only if it also converges to $p$ with the topology generated by taking $x$ as basepoint.
\end{proof}

We observe that $G$ acts by homeomorphisms on $G\cup M$ since 
$$h\widetilde{U}_G=\{hg\in G\mid gx\in \widetilde{U}\}\cup hU=\{g\in G\mid gx\in \widetilde{hU}\}\cup hU=\widetilde{hU}_G,$$
for any $h\in G$.

\begin{proof}[Proof of Theorem \ref{theo:compactification}]
	It is straightforward to see that the inclusions $G\hookrightarrow G\cup M$, $M\hookrightarrow G\cup M$ are topological embeddings.
	
	We now show that $G\cup M$ is compact. We observe that if $(g_n)$ is a convergent sequence in the sense of convergence group with attracting point $a\in M$, then $(g_n x)$ converges to $a$ for any $x\in T$, so $(g_n)$ converges to $a$ in the topology of $G\cup M$.
	Therefore, by definition of convergence group, any sequence $(h_n)\subset G$ has a converging sequence (in the sense of convergence groups) which converges (in the topology of $G\cup M$). Adding that $M$ is also compact, and that it is topologically embedded into $G\cup M$, we get that any sequence $g_n\in G\cup M$ has a converging subsequence. 
	
	For a sequence $(g_n)\subset G$ we have convergence to a point $p\in M$ if and only if $(g_n x)\subset T$ converges to the same point $p\in M$. Hence, by Proposition \ref{prop:convtobound}, random walks on $G$ converge almost surely to points in $M$.
\end{proof}

Whenever $G$ is a hyperbolic group, we can consider a finite set of generators $S$ and the Cayley graph $\Gamma(G,S)$. Then, we can add the Gromov boundary to $\Gamma(G,S)$, getting $\Gamma(G,S)\cup \partial G$, and then we can take the induced topology on $G\cup \partial G$. As changing the generating set $S$ induces a quasi-isometry, this topology on $G\cup \partial G$ does not depend on $S$, so it is well defined. The topology we have explained for a convergence group can be seen as an extension of Gromov's topology. Indeed, we have the following.
\begin{proposition}
	Let $G$ be an hyperbolic group, and assume its Gromov boundary $M$ has more than two points. Then, the topology we obtain on $G\cup M$ by considering $G$ acting as a convergence group on $M$ following the procedure explained in this section coincides with Gromov's topology.
\end{proposition}
\begin{proof}
	Both restrictions to $M$ and to $G$ have the same topology in both cases.
	Hence we only have to check if the sequences of $G$ converging to points in $M$ have the same limit in both topologies, as there are no sequences of elements of $M$ converging to elements of $G$.
	
	Consider $(g_n)_{n\in\N}\subset G$ converging to $\lambda\in M$ with the topology of convergence groups and assume the sequence does not converge to $\lambda$ in Gromov's topology.
	Then, given a finite set of generators $S$ there exists a subsequence $(g_{n_k})\subset (g_n)$ such that $(g_{n_k}\cdot \lambda)_e\le K$, where the Gromov product is taken with respect to the path metric on $\Gamma(G,S)$.
	Taking a convergent subsequence (in the sense of convergence groups) we get a subsequence $(h_i)\subset (g_{n_k})$ which, since it converges to $\lambda$ in the convergence group topology, has $\lambda$ as attracting point.
	Take $\alpha\in M$ different from the repelling point of $(h_i)$. Then $h_i\alpha$ converges to $\lambda$, so $(h_i\alpha\cdot \lambda)_e$ goes to infinity.
	Hence, by the reverse triangle inequality,
	$$K\ge (h_i\cdot \lambda)_e\ge \min((h_i\alpha\cdot h_i)_e,(h_i\alpha\cdot \lambda)_e)+C(\delta)=(h_i\alpha\cdot h_i)_e+C(\delta).$$
	Hence, $K\ge \d(e,[h_i\alpha,h_i])+C(\delta)=\d(h_i^{-1},[\alpha,e])+C(\delta)$. Therefore, $(h_i^{-1})$ converges to $\alpha$ in Gromov's topology. As $\alpha$ is can be any point in $M$ (except the repelling point of $(h_i)$) we get a contradiction, so $(g_n)$ converges to $\lambda$ in Gromov's topology.
	
	Assume now that $(g_n)$ converges to $\lambda\in M$ in Gromov's topology but that it does not converge to $\lambda$ in the convergence group topology. 
	Then there exists a subsequence $(g_{n_k})$ which converges to $\lambda'\neq \lambda$ with the convergence group topology, and by the previous paragraph, $(g_{n_k})$ converges to $\lambda'$ with Gromov's topology, giving us a contradiction.
\end{proof}

\subsection{The Dirichlet problem}
We have that $(M,\nu)$ is a $\mu$ boundary of $(G,\mu)$ with $\nu$ being the hitting measure of the random walk $(w_n x)\subset T$. That is, we have $w_n x\to p\in A\subset M$ with probability $\nu(A)$. Hence, for the random walk $(w_n)\subset G$ we have $w_n\to p\in A$ with probability $\nu(A)$, that is, the random walk $w_n$ converges pointwise to a random variable $w_\infty$ on $M$, with distribution $\nu$. We can define the hitting measures of the random walk starting at any $g\in G$ by
$$\nu_g(A):=\P[gw_\infty\in A\mid w_0=e]=\P[w_\infty\in g^{-1}A\mid w_0=e]=g\nu_e(A).$$

Given this setting, a frequent question is whether the \emph{Dirichlet problem at infinity} is solvable, that is, whether every continuous function $f:M\to \R$ admits a continuous extension to $G\cup M$ harmonic on $G$ with respect to the transition probability. For this we will, use the following theorem, a proof of which can be found in \cite[Theorem 20.3]{woess}.
\begin{theorem}
	The Dirichlet problem with respect a measure $\mu$ and a compactification $G\cup B$ of $G$ is solvable if and only if
	\begin{enumerate}
		\item the random walks $(w_n)$ converge almost surely to the boundary $B$;
		\item for the corresponding harmonic measures,
		$$\lim_{g\to p} \nu_g=\delta_p \: \text{weakly for every } p\in B.$$
	\end{enumerate}
\end{theorem}

We have already seen that the first requisite is satisfied. For the second one, we just have to observe that every sequence with $g_n\to p$ has a convergent subsequence (in the sense of convergence groups), for which $g_{n_k}\nu$ converges to $\delta_p$.
Therefore, $g_n \nu\to\delta_p$, since every subsequence has a convergent subsequence, and hence we cannot take a fully non converging subsequence. Hence, we get the following.

\begin{proposition}
	Assume $G$ is a non-elementary convergence group acting minimally on a metrizable space $M$. Then, the Dirichlet problem on $G\cup M$ solvable with respect to the topology defined above.
\end{proposition}

If $f:M\to \R$ is the continuous function on the boundary, the extension to $G$ is given by the Poisson formula
$$h(g)=\int_M f(p)g\nu.$$

\subsection{Strongly almost transitive actions}

Let $G$ be a second countable group acting measurably on a standard probability space $(X,\mathcal{B},\nu)$ in such a way that the action preserves the measure class of $\nu$ (that is, for all $g\in G$ and $A\in \mathcal{B}$ we have $\nu(A)=0\iff \nu(gA)=0$). 
We say that the action is \emph{strongly almost transitive} if, given a set $A\subset X$ such that $\nu(A)>0$ and $\epsilon>0$, there exists $g\in G$ such that $\nu(hA)>1-\epsilon$.
That is, the action is strongly almost transitive if every set of positive measure can be blown up to almost full measure. These actions were introduced by Jaworski in \cite{Jaworski}, where he proves the following theorem.

\begin{theorem}[Jaworski]
	Let $(M,\nu)$ be a $\mu$-boundary of $G$. Then, the action of $G$ on $(M,\nu)$ is strongly almost transitive.
\end{theorem}

\begin{corollary}
	Let $G$ be a non-elementary, minimal convergence group on a compact metrizable space $M$, and $\mu$ a measure on $G$ such that its support generates $G$. Then, there exists a measure $\nu$ such that the action on the probability space $(M,\nu)$ is strongly almost transitive. 
\end{corollary}

This result is in fact a consequence of previous results by Gekhtman, Gerasimov, Potyagailo and Yang \cite{ilya}.
We refer to the paper by Glasner and Weiss, \cite{glasner}, for a recompilation of some implications of having a nontrivial strongly almost transitive action. We write here one of the consequences explained in that paper, which we find particularly interesting.

\begin{corollary}[of {\cite[Proposition 4.3]{glasner}}]
	Let $G$ be non-elementary, minimal convergence group on a compact metrizable space $M$. Then, there is no non-constant Borel measurable equivariant map $\phi:M\to Z$, where $(Z,d)$ is a separable metric space on which $G$ acts by isometries (that is, $M$ is \emph{ergodic with isometric coefficients}).
\end{corollary}

\subsection{$F_\mu$-proximality}
Given a measure $\mu$ on a discrete countable group $G$, we can define the Ces\`aro averages $\mu_n:=\frac{1}{n}(\mu+\mu^2+\ldots+\mu^n)$. We say that a compact metric $G$-space $X$ is \emph{$F_\mu$-proximal} if for each $x,y\in X$, $\mu_n\{g:\d(gx,gy)>\epsilon\}\to 0$ as $n\to \infty$ for any $\epsilon>0$. Furstenberg introduced this notion in \cite{Furstenberg3}, where he also shows (among other equivalences, see Theorem 14.1 of that same article; see also \cite[Theorems 8.4 and 8.5]{glasner} for a slightly larger list) that $X$ is $F_\mu$-proximal if and only if for any $\mu$-stationary Borel probability measure $\nu$ on $X$, the couple $(X,\nu)$ is a $\mu$-boundary of $G$. This is indeed the case for convergence groups, so we get the following result, as it is shown in \cite{ilya}.
\begin{corollary}
	Let $G$ be a non-elementary, minimally convergence group on a compact metrizable space $M$, and $\mu$ a measure on $G$ such that its support generates $G$. Then, $M$ is $F_\mu$-proximal.
\end{corollary}

%    Bibliographies can be prepared with BibTeX using amsplain,
%    amsalpha, or (for "historical" overviews) natbib style.
\bibliographystyle{amsplain}
\bibliography{rw}{}

\providecommand{\bysame}{\leavevmode\hbox to3em{\hrulefill}\thinspace}
\providecommand{\MR}{\relax\ifhmode\unskip\space\fi MR }
% \MRhref is called by the amsart/book/proc definition of \MR.
\providecommand{\MRhref}[2]{%
  \href{http://www.ams.org/mathscinet-getitem?mr=#1}{#2}
}
\providecommand{\href}[2]{#2}
\begin{thebibliography}{10}

\bibitem{Bowditch2}
B.~H. Bowditch, \emph{A topological characterization of hyperbolic groups}, J.
  Amer. Math. Soc. \textbf{11} (1998), no.~3, 643--667.

\bibitem{Bowditch}
\bysame, \emph{Convergence groups and configuration spaces}, Geometric group
  theory down under (Canberra, 1996) (W.~D.~Neumann J.~Cossey, C. F.~Miller and
  M.~Shapiro, eds.), 1999, pp.~23--54.

\bibitem{spacesnonpositivecurvature}
M.~Bridson and A.~Haefliger, \emph{Metric spaces of non-positive curvature},
  Grundlehren der Mathematischen Wissenschaften, vol. 319, Springer-Verlag,
  1999.

\bibitem{Furstenberg}
H.~Furstenberg, \emph{{Random walks and discrete subgroups of Lie groups}},
  Advances in probability and related topics (P.~Ney, ed.), vol.~1, 1971,
  pp.~3--63.

\bibitem{Furstenberg3}
\bysame, \emph{Boundary theory and stochastic processes on homogeneous spaces},
  Proceedings of Symposia in Pure Mathematics (C.~C. Moore, ed.), vol. 26:
  Harmonic Analysis on Homogeneous Spaces, Amer. Math. Soc., 1973,
  pp.~193--229.

\bibitem{Kleinian}
F.~W. Gehring and G.~J. Martin, \emph{{Discrete quasiconformal groups I}},
  Proc. Londom Math. Soc \textbf{55} (1987), 331--358.

\bibitem{ilya}
I.~Gekhtman, V.~Gerasimov, L.~Potyagailo, and A.~Yang, \emph{Martin boundary
  covers floyd boundary}, arXiv:1708.02133 (2017).

\bibitem{glasner}
E.~Glasner and B.~Weiss, \emph{Weak mixing properties for non-singular
  actions}, Ergod. Th. and Dynam. Sys. \textbf{36} (2016), 2203--2217.

\bibitem{Jaworski}
W.~Jaworski, \emph{Strongly approximately transitive group actions, the
  choquet-deny theorem, and polynomial growth}, Pacific J. Math \textbf{165}
  (1994), 115--129.

\bibitem{Kaimonovich3}
V.~A. Kaimanovich, \emph{Boundaries of invariant markov operators: The
  identification problem}, Proceedings of the Conference on Algebraic and
  Number Theoretic Aspects of Ergodic Theory (Warwick, 1994) (M.~Pollicot and
  K.~Schmidt, eds.), London Math. Soc. Lecture Note Series, Cambridge
  University Press, 1996.

\bibitem{Kaimonovich2}
\bysame, \emph{The {P}oisson boundary of hyperbolic groups}, Comptes Rendus de
  l'Académie des Sciences \textbf{1} (1997), 59--64.

\bibitem{Kaimonovich}
\bysame, \emph{The {P}oisson boundary for groups with hyperbolic properties},
  Annals of Mathematics \textbf{152} (2000), 659--692.

\bibitem{Maher}
J.~Maher, \emph{Linear progress in the complex of curves}, Trans. Amer. Math.
  Soc. \textbf{362} (2010), no.~6, 2963--2991.

\bibitem{Tiozzo2}
J.~Maher and G.~Tiozzo, \emph{Random walks on weakly hyperbolic groups},
  Journal f{\"u}r die reine und angewandte Mathematik (Crelles Journal)
  \textbf{2018} (2018), no.~742, 187--239.

\bibitem{Tiozzo3}
\bysame, \emph{Random walks, {WPD} actions, and the {C}remona group},
  arXiv:1807.10230 (2018).

\bibitem{binsun}
B.~Sun, \emph{A dynamical characterization of acylindrically hyperbolic
  groups}, Algebraic \& Geometric Topology \textbf{19} (2019), no.~4,
  1711--1745.

\bibitem{Tiozzo1}
G.~Tiozzo, \emph{Sublinear deviation between geodesics and sample paths}, Duke
  Math. J. \textbf{164} (2015), no.~3, 511--539.

\bibitem{Tukia2}
P.~Tukia, \emph{Convergence groups and {G}romov's metric hyperbolic spaces,},
  New Zealand J. Math \textbf{23} (1994), 157--187.

\bibitem{Tukia}
\bysame, \emph{Conical limit points and uniform convergence groups}, J. reine
  angew. Math \textbf{501} (1998), 71--98.

\bibitem{woess}
W.~Woess, \emph{Random walks on infinite graphs and groups}, 2nd ed., Cambridge
  University Press, 2000.

\end{thebibliography}
%    Insert the bibliography data here.

\end{document}